\documentclass{articlefederico2017}

\usepackage{amsfonts}
\usepackage{amsmath}
\usepackage{amssymb}
\usepackage{amsthm}
\usepackage{upref}
\usepackage[colorlinks,final,backref=page,hyperindex]{hyperref}
\usepackage{mathtools}
\usepackage{float,tikz}
\usepackage{cleveref}
\usepackage{ bbold }
\usepackage{soul}

\hypersetup{citecolor={blue}}

\newtheorem{theorem}{Theorem}[section] \numberwithin{equation}{section}
\newtheorem{corollary}[theorem]{Corollary}
\newtheorem{proposition}[theorem]{Proposition}

\newtheorem{conjecture}[theorem]{Conjecture}
\newtheorem{exam}[theorem]{Example}
\newenvironment{example}{\begin{exam}\rm}{\end{exam}}
\newtheorem{rem}[theorem]{Remark}
\newenvironment{remark}{\begin{rem}\rm}{\end{rem}}
\newtheorem{definition}[theorem]{Definition}

\newtheorem{innercustomgeneric}{\customgenericname}
\providecommand{\customgenericname}{}
\newcommand{\newcustomtheorem}[2]{%
	\newenvironment{#1}[1]
	{%
		\renewcommand\customgenericname{#2}%
		\renewcommand\theinnercustomgeneric{##1}%
		\innercustomgeneric
	}
	{\endinnercustomgeneric}
}

\newcustomtheorem{customthm}{Theorem}
\newcustomtheorem{customprop}{Proposition}

\renewcommand\emptyset{\varnothing}

\newcommand\commentout[1]{}

\newcommand\vol{\operatorname{vol}} 
\newcommand\conv{\operatorname{conv}}

\newcommand\ehr{\operatorname{ehr}} 
\newcommand\Ehr{\operatorname{Ehr}}

\newcommand\Aut{\operatorname{Aut}}

\newcommand\ZZ{\mathbb{Z}}

\newcommand\RR{\mathbb{R}}
\newcommand\KK{\mathbb{K}}

\newcommand\cA{\mathcal{A}}

\newcommand\cC{\mathcal{C}}

\newcommand\cM{\mathcal{M}}
\newcommand\cP{\mathcal{P}}

\newcommand\Zono{\mathcal{Z}}

\newcommand\mc[1]{\mathcal{#1}}
\newcommand\Poly{\mathcal{P}}

\newcommand\bV{\mathbf{V}}

\newcommand\be{\mathsf{e}}

\newcommand\bu{\mathbf{u}}
\newcommand\bv{\mathbf{v}}

\newcommand\bx{\mathbf{x}}

\newcommand\bzero{\mathbf{0}}

\newcommand\wC{\widetilde{C}}

\newcommand\wX{\widetilde{X}}
\newcommand\wY{\widetilde{Y}}
\newcommand\wZ{\widetilde{Z}}
\newcommand\walpha{\widetilde{\alpha}}

\begin{document}

\title{Oriented Matroid Circuit Polytopes}

\author{
\textsf{Laura Escobar\footnote{\noindent \textsf{University of California, Santa Cruz; lauraescobar@ucsc.edu.}}}
\textsf{ and Jodi McWhirter\footnote{\noindent \textsf{jodi.mcwhirter@wustl.edu.}}}
}

\maketitle

\begin{abstract}
Matroids give rise to several natural constructions of polytopes.
Inspired by this, we examine polytopes that arise from the signed circuits of an oriented matroid. We give the dimensions of these polytopes arising from graphical oriented matroids and their duals. Moreover, we consider polytopes constructed from cocircuits of oriented matroids arising from the positive roots of type $A$ root systems. We give an explicit description of their face structure and determine the Ehrhart series. 
We also study an action of the symmetric group on these polytopes, giving a full description of the subpolytopes fixed by each permutation.
These type $A$ polytopes are graphic zonotopes, are polar duals of symmetric edge polytopes, and appear in Stapledon’s paper introducing equivariant Ehrhart theory.
\end{abstract}


\section{Introduction}

Polytopes have been a key object in the study of matroids. There are various ways that one can construct a polytope from a matroid.
For example, Barahona and Gr\"otschel defined in \cite{BG} the cycle polytope of a matroid.
Since matroid polytopes are arithmetically interesting, see e.g.~\cite{deloerahawskoeppe}, Matthias Beck posed the question in 2021 to study the arithmetic of polytopes constructed from the indicator vectors of the signed circuits of an oriented matroid.
We study these oriented matroid circuit (OMC) polytopes. Using the circuit axioms, we show that every circuit yields a vertex of the corresponding OMC polytope and that every OMC polytope is centrally symmetric. Turning to graphical and co-graphical oriented matroids, we determine the dimension of the corresponding OMC polytopes. We also describe the effect that some special graph edges, namely bridges and loops, have on the OMC polytope. 
Integral combinations of circuits of an oriented matroid have been arranged into a lattice by \cite{antisymmetricflows} to generalize flows. We take the convex hull of the indicator vectors of circuits, which are elements of this lattice.

We give more detailed information on a family of OMC polytopes that arise from the classical type $A_{n}$ root system, which we call $\Poly_{n}$. Aside from being an OMC polytope, $\Poly_{n}$ appears in other families of polytopes. For instance, we show that it is affinely isomorphic to the graphic zonotope $\Zono_{C_{n+1}}$, where $C_{n+1}$ is the cycle graph on $n+1$ nodes. It is also a tropical unit ball, as seen in \cite{criadotropicalbisectors}, which is the polar dual of the symmetric edge polytope for the complete graph $K_{n+1}$. We describe the face structure of $\Poly_{n}$.

We also study the Ehrhart theory of $\Poly_{n}$. Moreover, the symmetric group $S_{n+1}$ action on the graph $K_{n+1}$ induces an $S_{n+1}$ action on the vertices of $\Poly_{n}$. We describe each subpolytope $\Poly_{n}^\sigma$ of points of $\Poly_{n}$ that are fixed by a $\sigma\in S_{n+1}$. These fixed polytopes fit in the larger picture of equivariant Ehrhart theory, introduced by Stapledon in \cite{stapledonequivariant}. 
 It is of recent interest to compute the equivariant Ehrhart theory of polytopes that admit a group action, see e.g.  \cite{ardilaschindlervindas,ardilasupinavindas,ClaHigKol,equivarianttechniques}, in part due to a conjecture by Stapledon \cite[Conjecture 12.1]{stapledonequivariant}.
In fact, $\Poly_{n}$ also appears in \cite{stapledonequivariant}, where Stapledon connects its equivariant Ehrhart theory to the cohomology of the toric variety of the permutahedron. We use the fixed polytopes of $\Poly_{n}$ to compute the equivariant Ehrhart theory in a different way.

A natural question is to study these polytopes for other Coxeter types. First, we observe that types $B_n$ and $C_n$ yield the same polytopes: since signed circuits encode the positive and negative coefficients of linear dependencies of vectors, the set of signed circuits of $\cM(B_n)$ and $\cM(C_n)$ is the same, resulting in the same OMC polytopes. While it would be interesting to study the OMC polytopes of these other Coxeter types, these polytopes are very high-dimensional in types $B_n/C_n$ and $D_n$, as referenced in Example \ref{eg:typeB}, which leads to some challenges.

We now give an outline of the paper. In \Cref{sec_ocm_def}, we introduce some background on oriented matroids and define oriented matroid circuit (OMC) polytopes in Definition \ref{def:omcpolytope}.
Proposition \ref{thm_graph_dim} and Proposition \ref{thm_dualgraph_dim} give the dimension of the OMC polytope for graphical and co-graphical oriented matroids, respectively.

\Cref{sec_dual_a} examines the OMC polytope $\Poly_{n}$ coming from the family of oriented matroids dual to those determined by the positive roots of the type $A_{n}$ root system. This family is also co-graphical, being dual to the graphical oriented matroids determined by the complete graph $K_{n+1}$. In \Cref{sec_other_appearances}, we find $\Poly_{n}$ in the context of graphic zonotopes \cite{Gru} and symmetric edge polytopes \cite{sepgentoregmats}. \Cref{thm_faces} and Corollary \ref{cor_face_lattice} in \Cref{sec_face_structure} connect the face structure of $\Poly_{n}$ to the Boolean lattice on $[n+1]$.  

In \Cref{sec_ehrhart}, we study the Ehrhart theory of $\Poly_{n}$, which is given specifically in Proposition \ref{ehrdA}. 
Inspired by this case as well as computational evidence using Finschi's catalog oriented matroids \cite{FinschiCatalogOM}, we conjecture in Conjecture \ref{conj} that the $h^*$-vector of every oriented matroid circuit polytope is log-concave.
Sections~\ref{sec_eet} and \ref{sec_2eet} concern equivariant Ehrhart theory. In \Cref{sec_eet}, we fully describe the $S_{n+1}$-action on $\Poly_n$ and, in \Cref{thm_fixedpolytopes}, the polytopes $\Poly_n^\sigma$ fixed by each $\sigma\in S_{n+1}$. Using these fixed polytopes, Corollary \ref{itsapolynomial} gives explicitly the equivariant $H^*$-series of $\Poly_n$ under our $S_{n+1}$ action. 
Finally, in \Cref{sec_2eet}, we compare the fixed polytopes of our $S_{n+1}$ action on $\Poly_{n}$ to the fixed polytopes of a different group action coming from graphic zonotopes.



\section{Oriented matroid circuit polytopes}\label{sec_ocm_def}

We begin by giving the background on oriented matroids, following \cite[Section 3]{orientedmatroids}.
A \textbf{signed set} $\widetilde{X}$ is a set $X$ together with a partition $(X^+,X^-)$ of $X$.
The \textbf{opposite} of $\widetilde{X}$ is the signed set $-\wX=((-X)^+,(-X)^-)$ with $(-X)^+:=X^-$ and $(-X)^-:=X^+$.

\begin{definition}\label{def:OM}
A set $\cC$ of signed sets is the set of \textbf{circuits} of an \textbf{oriented matroid} if the following hold:
\begin{itemize}
		\item (C0) The empty signed set $(\emptyset,\emptyset)$ is not a circuit
		\item (C1) If $\wX$ is a circuit, then so is $-\wX$
		\item (C2) 	For all $\wX,\wY\in\mc{C}$, if ${X}\subseteq{Y}$, then $\wX=\wY$ or $\wX=-\wY$
		\item (C3) If $\wX,\wY$ are circuits with $\wX\neq -\wY$ and $e\in X^+\cap Y^-$, then there is a third circuit $\wZ$ such that 
			$Z^+\subseteq (X^+\cup Y^+)\setminus\{e\}$ and 
			$Z^-\subseteq (X^-\cup Y^-)\setminus\{e\}$.
	\end{itemize}
\end{definition}
In the literature, these ``circuits'' are often referred to as ``signed circuits''; since we are rarely dealing with unsigned circuits, we will just use ``circuits.''

Representable oriented matroids are given by lists of vectors in a vector space as follows. 
Let $\KK$ be a totally ordered field.
Given a matrix $M\in \KK^{d\times n}$ with column vectors $\bv_1,\ldots,\bv_n$, one associates an oriented matroid as follows.
A linear dependence of $M$ is $(\lambda_1,\ldots,\lambda_n)\in\KK^n$ such that $\sum_{i=1}^n\lambda_i\bv_i=\bzero$.
A linear dependence is minimal if there is no linear dependence $(\lambda'_1,\ldots,\lambda'_n)$ such that
$$\{i\in[n]\mid \lambda'_i\neq 0\}\subsetneq \{i\in[n]\mid \lambda_i\neq 0\}.$$
Given a minimal linear dependence $(\lambda_1,\ldots,\lambda_n)\in\KK^n$ we associate the signed set $\widetilde{X}=(X^+,X^-)$ given by
\begin{equation*}
    X^+=\{i\in[n]\mid \lambda_i>0\},\qquad
    X^-=\{i\in[n]\mid \lambda_i<0\}.
\end{equation*}
The circuits of the oriented matroid $\cM(M)$ consist of the signed sets associated to minimal linear dependencies.
Moreover, an oriented matroid $\cM$ is \textbf{representable over $\KK$} if $\cM=\cM(M)$ for some matrix $M$ with entries in $\KK$.

For example, consider the matrix
\begin{align*}
	M = \begin{pmatrix}
		0  & -1 & -1 & 0 & 0 & 0 \\
		14 & -1 & -9 & 0 & 0 & 0 \\
		1  &  5 & -1 & 1 & 0 & 1
	\end{pmatrix}.
\end{align*}
We use the numbers $1,2,\dots,6$ to denote the columns of the matrix and, to simplify notation, often use, for instance, $136$ to denote $\{1,3,6\}$. The circuits of $\cM(M)$ are $(5,\emptyset)$, $(\emptyset, 5)$, $(4,6)$, $(6,4)$, $(134,2)$, $(2,134)$, $(136,2)$, $(2,136)$. 
Another way to denote circuits is through strings of $+$s, $-$s, and $0$s. For example, the circuit $(134,2)$ would be written as $(+,-,+,+,0,0)$.

As a second example, consider the matrix 
\begin{align*}
	M = \begin{pmatrix}
		1 &  1 & 1 &  1 & 0 & 0  \\
		1 & -1 & 0 &  0 & 1 &  1 \\
		0 &  0 & 1 & -1 & 1 & -1
	\end{pmatrix},
\end{align*}
whose columns are the type $D_3$ positive roots. The oriented matroid $\mc{M}(\mc{D}_3)$ corresponding to the matrix $M$ has 14 circuits, including $(1,36)$ and $(36,45)$.

Let us describe the notion of a dual oriented matroid in the context of representable matroids.
The oriented matroid $\cM^*(M)$ which is dual to $\cM(M)$ is given as follows.
Given $\bu\in\mathrm{rowsp}(M)$, the \textbf{signed support} of $\bu$ is the signed set $\widetilde{X}=(X^+,X^-)$ where
    $$
    X^+=\{i\in [n] \mid u_i>0\},\qquad
    X^-=\{i\in [n] \mid u_i<0\}.
    $$
The circuits of $\cM^*(M)$ are the inclusion-minimal non-empty signed supports of the vectors in $\mathrm{rowsp}(M)$,
see e.g.\ \cite[Lemmas 4.1.34 and 4.1.38]{deloerahemmeckekoeppe}.

\subsection{Oriented matroid circuit polytopes}

Having introduced the necessary background, we are now ready to define the polytopes we study.
Let $\mc{M}$ be an oriented matroid with circuit set $\mc{C}$ and \textbf{ground set} $E$, i.e. for each $\widetilde{X}\in \mc{C}$, $X\subseteq E$.
For each signed set $\wX$ with $X\subseteq E$, define the \textbf{signed incidence vector} $\textbf{v}_{\wX}\in\RR^E$ by
\begin{align*} (\textbf{v}_{\widetilde{X}})_j := 
	\begin{cases}
		0, & j \notin {X} \\
		1, & j\in X^+ \\
		-1, & j\in X^-
	\end{cases}
.\end{align*}

We now introduce the main definition of our paper.

\begin{definition}\label{def:omcpolytope}
The \textbf{oriented matroid circuit (OMC) polytope} $P_\cM\subseteq\RR^E$ associated to the matroid $\cM$ with ground set $E$ and nonempty circuit set $\cC$ is the convex hull of all such $\textbf{v}_{\wX}$, that is,
\begin{align*}
	P_\cM = \conv\left\{\textbf{v}_{\wX} \mid \wX\in\cC\right\}.
\end{align*}
\end{definition}

\begin{remark} \label{rem:reverseorientation}
Let $\cM$ be an oriented matroid with ground set $E$.
By \cite[Remark~3.2.3]{orientedmatroids}, given $A\subseteq E$ we can consider the oriented matroid $\cM'$ obtained by reversing the orientation of the elements of $A$.
Note that the linear map
	$$
	\RR^E\to\RR^E,
	\qquad 
	x_e\mapsto
	\begin{cases}
	-x_e,&e\in A\\
	x_e,&e\notin A
	\end{cases}
	$$
witnesses that $P_\cM$ and $P_{\cM'}$ are unimodularly equivalent polytopes.
Also, note that if $\cM_1$ and $\cM_2$ are oriented matroids with circuits $\cC_1,\cC_2$, then 
\[
P_{\cM_1\oplus \cM_2}=\conv \left((P_{\cM_1}\times\{\bzero\}) \cup (\{\bzero\}\times P_{\cM_2}\right)
,\] 
where, $\cM_1\oplus \cM_2$ denotes the direct sum of the matroids.
\end{remark}

We start by describing the vertices of OMC polytopes.

\begin{proposition}\label{vertices}
	The vertices of $P_\cM$ are precisely those $\textbf{v}_{\wX}$ such that $\wX$ is a circuit of $\cM$.
\end{proposition}
\begin{proof}
	Suppose $\cM$ is an oriented matroid, $\cC$ is its set of circuits, $V = \{\textbf{v}_{\wX}\mid \wX\in\cC\}$, and $P$ is the OMC polytope. We would like to show that each $\textbf{v}_{\wX}\in V$ is a vertex of $P$.
	
	Let $\wX\in\mc{C}$ be a circuit with $X^+ = \{a_1,\dots,a_k\}$ and $X^- = \{b_1,\dots,b_l\}$. 
	Now, consider $x\in P$. If we take the dot product of $x$ with the direction vector $\textbf{v}_{\wX}$, we get
	\begin{align*}
	    \textbf{v}_{\wX}\cdot x 
	    &= (x_{a_1} + \dots + x_{a_k}) - (x_{b_1} + \dots + x_{b_l}) \\
	    &\leq k + l
	\end{align*}
	since $|x_m|\leq 1$. If $x = \textbf{v}_{\wX}$, then $\textbf{v}_{\wX}\cdot x = \textbf{v}_{\wX}\cdot \textbf{v}_{\wX} = k+l$. In particular, the linear functional $\mathbf v_{\wX}\cdot x$ attains its maximum over $P$ at $x=\mathbf v_{\wX}$ and thus lies on a face of $P$. Now, consider $\wY\in\mc{C}$ and suppose $\textbf{v}_{\wX} \cdot \textbf{v}_{\wY} = k+l$ as well. It follows that $X^+\subseteq Y^+$ and $X^-\subseteq Y^-$. Thus ${X}\subseteq {Y}$. By circuit axiom (C2), this implies $\wX = \wY$. In particular, $\textbf{v}_{\wX}$ is a vertex of $P$.
\end{proof}

The following is immediate.

\begin{corollary}
For any oriented matroid $\cM$, the polytope $P_\cM$ is centrally symmetric, i.e. $x\in P_\cM$ implies $-x\in P_\cM$.
\end{corollary}

\begin{remark}\label{rem_dimension_general}
Note that since $P_\cM$ is centrally symmetric, the origin lies in its relative interior so that its dimension equals the dimension of the vector space spanned by its vertices. 
Some facts are known about this dimension, see e.g.~\cite[Exercise~4.45]{orientedmatroids}. As a consequence, for any oriented matroid with ground set $E$ and rank $r$, $\dim(P_\cM)\ge |E|-r$. 
However, it is an open problem to find the dimension of this vector space for general oriented matroids.
\end{remark}

The following proposition implies that the only lattice points on the edges of an OMC polytope are vertices. 

\begin{proposition}
For any oriented matroid $\cM$ and distinct $\wX,\wY\in\cC$ we have 
$$
(\textbf{v}_{\wX},\textbf{v}_{\wY})\cap \ZZ^E\subseteq \{0\}
,$$
where $(\textbf{v}_{\wX},\textbf{v}_{\wY}):=\{\lambda\textbf{v}_{\wX}+(1-\lambda)\textbf{v}_{\wY}\mid 0< \lambda< 1\}$.
\end{proposition}

\begin{proof}
Suppose that $(\textbf{v}_{\wX},\textbf{v}_{\wY})\cap \ZZ^E\neq \varnothing$.
Note that since $\textbf{v}_{\wX}$ and $\textbf{v}_{\wY}$ are lattice points in the cube $[-1,1]^E$, the existence of such a lattice point implies $|(\textbf{v}_{\widetilde{X}})_i|=|(\textbf{v}_{\widetilde{Y}})_i|$ for all $i$, i.e., $X=Y$.
It follows from Axiom (C2) in Definition \ref{def:OM} that  $\wX=-\wY$. We conclude that $(\textbf{v}_{\wX},\textbf{v}_{\wY})\cap \ZZ^E=\{0\}$, as desired.
\end{proof}

We remark that the analogous statement does not hold for larger collections of vertices as the following example shows.

\begin{example}\label{eg:notTerminal}
Let $\cM=\cM(M)$, where 
\begin{equation*}
M = \begin{pmatrix}
		1  & 0 & 0 & 10 & 1 & -1 & -9 \\
		0 & 1 & 0 & 1 & 1 & -2 & -2 \\
		0& 0 & 1  &  1 & 0 & 1 & -1
	\end{pmatrix}.
\end{equation*}
Label the columns of $M$ by $\textbf{w}_1,\ldots,\textbf{w}_7$ and let $\wX_1,\ldots, \wX_4$ be the circuits associated to the following minimal linear dependencies
\begin{align*}
&10\textbf{w}_1+\textbf{w}_2+\textbf{w}_3-\textbf{w}_4=\bzero, \qquad 6\textbf{w}_2-10\textbf{w}_5-\textbf{w}_6-\textbf{w}_7=\bzero,\\
 &\textbf{w}_2-\textbf{w}_3+\textbf{w}_5+\textbf{w}_6=\bzero, \quad\text{and}\quad -\textbf{w}_1+\textbf{w}_2+\textbf{w}_4+\textbf{w}_7=\bzero.
\end{align*}
We then have that $\frac{1}{4}(\textbf{v}_{\wX_1}+\cdots+\textbf{v}_{\wX_4})=(0,1,0,0,0,0,0)\in P_\cM\cap\ZZ^7$ is not a vertex of the polytope.

\end{example}


\subsection{Graphical oriented matroid circuit polytopes}\label{sec_graphical}

In this section we concentrate on graphical oriented matroids. 
Given a graph $G=(V,E)$ let $\widetilde G$ be a directed graph obtained by orienting each edge of $G$.
The circuits of the oriented matroid $\cM(\widetilde{G})$ are obtained from the cycles of $G$ as follows.
Given a cycle $C$ of $G$, choose a way to traverse $C$. We obtain the circuits $\wC=(C^+,C^-)$ and $-\wC$ by setting 
    \begin{align*}
    C^+&:=\{e\in E\mid e \text{ is traversed following its orientation}\}
    \\
    C^-&:=\{e\in E\mid e \text{ is traversed opposite to its orientation}\}
    .\end{align*}
The circuits of $\cM(\widetilde{G})$ are obtained by constructing the circuits above for all cycles of $G$.
It is well known that $\cM(\widetilde{G})$ is represented over any ordered field by its incidence matrix.

For example, consider the graph in 
\begin{figure}[h!]
	\centering
	\includegraphics[width=0.17\textwidth]{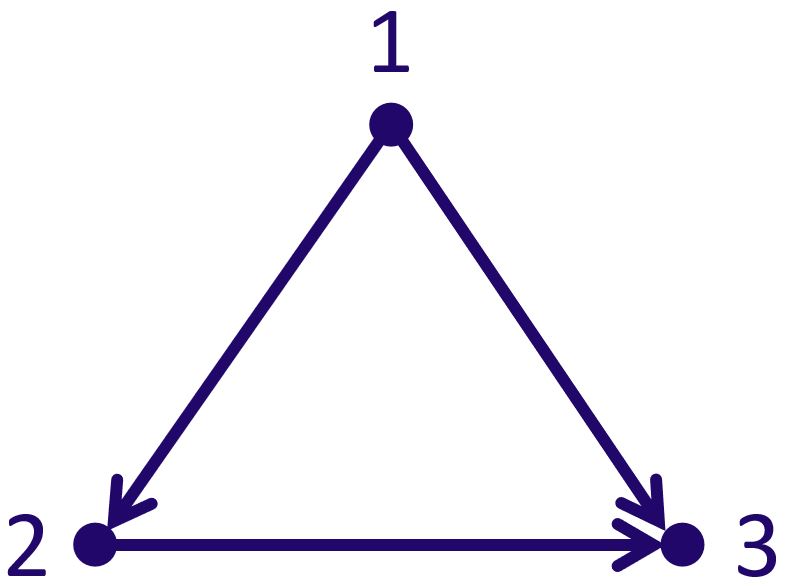}
	\caption{The complete graph $\widetilde{K_3}$ on 3 nodes, oriented so that the arrow points at the larger number.}
	\label{a2graph}
\end{figure}
\Cref{a2graph}. The oriented matroid $\cM(\widetilde{K_3})$ has two circuits: $(\{\vec{12},\vec{23}\},\{\vec{13}\})$ and $(\{\vec{13}\},\{\vec{12},\vec{23}\})$, which can also be represented, under the indexing $(\vec{12},\vec{13},\vec{23})$, as $(+,-,+)$ and $(-,+,-)$. The OMC polytope $P_{\cM(\widetilde{K_3})}$ is therefore a line segment in $\RR^3$.

We now give the dimension of $P_{\cM(\widetilde{G})}$.

\begin{proposition}\label{thm_graph_dim}
	Let $G = (V,E)$ be a graph and $\widetilde{G}$ be a directed graph obtained by orienting each edge of $G$. 
	The dimension of $P_{\cM(\widetilde{G})}$ is $|E|-|V|+k$, where $k$ is the number of connected components of $G$.
\end{proposition}

\begin{proof}
As discussed in \Cref{rem_dimension_general}, we want to compute the dimension of the vector space spanned by the vertices of $P_{\cM(\widetilde{G})}$.
By \cite[Exercise~4.45]{orientedmatroids}, this dimension is $|E|-r$, where $r$ is the rank of $\cM(\widetilde{G})$.
The proposition follows, since the rank of $\cM(\widetilde{G})$ equals the number of edges in a spanning forest, i.e. $r=|V|-k$.
\end{proof}

\subsection{Graphical oriented matroid cocircuit polytopes}

Let $G=(V,E)$ be a graph. 
The following can be found in \cite[\S~1.1]{orientedmatroids}. 
A \textbf{minimal cut} of $G$ is a partition $V=V_1\sqcup V_2$ such that deleting all edges in $G$ between $V_1$ and $V_2$ increases by one the number of connected components of $G$.
Let $\widetilde G$ be a directed graph obtained by orienting each edge of $G$.
We now describe how to obtain a matroid $\cM^*(\widetilde{G})$, which is dual to the matroid  $\cM(\widetilde{G})$ of \Cref{sec_graphical}, from the minimal cuts of $G$. 
The circuits of the oriented matroid $\cM^*(\widetilde{G})$ are obtained from the minimal cuts of $G$ as follows.
Given a partition $V=V_1\sqcup V_2$, we obtain the signed sets $\wX=(X^+,X^-)$ and $-\wX$ by setting 
    \begin{align*}
    X^+&:=\{e\in E\mid e \text{ goes from } V_1 \text{ to } V_2\}
    \\
    X^-&:=\{e\in E\mid e \text{ goes from } V_2 \text{ to } V_1\}
    .\end{align*}
The circuits of $\cM^*(\widetilde{G})$ are obtained by constructing the circuits above for all minimal cuts of $G$.

\begin{figure}[h!]
	\centering
	\includegraphics[width=0.2\textwidth]{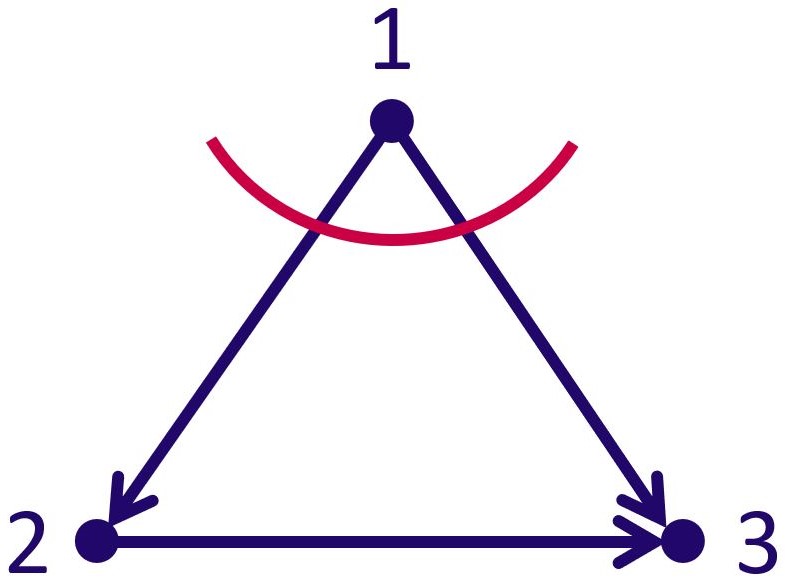}
	\caption{A cut of $\widetilde{K_3}$ with $V_1 = \{1\}$ and $V_2 = \{2,3\}$. This depicts the cocircuits $(\{\vec{12},\vec{13}\},\emptyset)$ and $(\emptyset,\{\vec{12},\vec{13}\})$.}
	\label{a2graph1cut}
\end{figure}
To continue the example of $\widetilde{K_3}$, the cocircuits of $\cM(\widetilde{K_3})$ are $(\{\overrightarrow{12},\overrightarrow{13}\},\emptyset)$, $(\emptyset,\{\overrightarrow{12},\overrightarrow{13}\})$, $(\{\overrightarrow{12}\},\{\overrightarrow{23}\})$, $(\{\overrightarrow{23}\},\{\overrightarrow{12}\})$, $(\{\overrightarrow{13},\overrightarrow{23}\},\emptyset)$, and $(\emptyset,\{\overrightarrow{13},\overrightarrow{23}\})$, which can also be written as $(+,+,0)$, $(-,-,0)$, $(+,0,-)$, $(-,0,+)$, $(0,+,+)$, and $(0,-,-)$. The OMC polytope $P_{\cM^*(\widetilde{K_3})}$ is a hexagon in $\RR^3$, pictured in \Cref{2withcoords}.

\begin{proposition}\label{thm_dualgraph_dim} 
	Let $G = (V,E)$ be a graph and $\widetilde{G}$ be a directed graph obtained by orienting the edges of $G$. 
	The dimension of $P_{\cM^*(\widetilde{G})}$ is $|V|-k$, where $k$ is the number of connected components of $G$.
\end{proposition}

\begin{proof}
The incidence matrix $A$ of $\widetilde{G}$ represents $\cM(\widetilde{G})$ over any ordered field. 
It follows that the circuits $\cC^*$ of $\cM^*(\widetilde{G})$ are the inclusion-minimal non-empty signed supports of the vectors in $\mathrm{rowsp}(A)$.
We wish to compute the dimension of the vector space spanned by $\{\textbf{v}_{\wX} \mid \wX\in\cC^*\}$.

Let us show that $\mathrm{rowsp}(A)=\mathrm{span}\{\textbf{v}_{\wX} \mid \wX\in\cC^*\}$.
Note that the row of $A$ corresponding to the node $i\in V$ is precisely $\bv_{\wY_i}$, where $\wY_i$ denotes the signed set corresponding to the partition $\{i\}\sqcup (V\setminus\{i\})$.
Let $V=V_1\sqcup V_2$ be a minimal cut of $G$ and $\wX$ the corresponding circuit.
We claim that 
    \begin{equation}\label{eq_bond_expansion}
        \bv_{\wX}=\sum_{i\in V_1} \bv_{\wY_i}.    
    \end{equation}
Suppose that $e=(j,m)\notin X$ so that $(\bv_{\wX})_e=0$ and note that either $j,m\in V_1$ or  $j,m\in V_2$.
In the first case, we have 
    $$
    \sum_{i\in V_1} (\bv_{\wY_i})_e
    =(\bv_{\wY_j})_e+(\bv_{\wY_m})_e=1-1=0.
    $$
The second case follows from $(\bv_{\wY_i})_e=0$ for all $i\in V_1$.
Next, suppose that $e=(j,m)\in X^+$ and note that  
    $$
    (\bv_{\wX})_e=1=(\bv_{\wY_j})_e=\sum_{i\in V_1} (\bv_{\wY_i})_e.
    $$
The case $e\in X^-$ follows from similar reasoning.
It follows that $\mathrm{rowsp}(A)\supseteq\mathrm{span}\{\textbf{v}_{\wX} \mid \wX\in\cC^*\}$.

We now prove the opposite inclusion. 
Fix the node $i\in V$ and let $D_1,\ldots,D_a$ be the connected components of $G-i$ that lie inside the connected component of $i$. For each $j=1,\ldots,a$ let $\wX_j$ be the circuit corresponding to the minimal cut $(V\setminus D_j)\sqcup D_j$. We claim that
\[
\bv_{\wY_i} = \sum_{j=1}^a \bv_{\wX_j}.
\]
Since every edge incident to $i$ is incident to a vertex in exactly one $D_j$, it follows that
\[
(\bv_{\wY_i})_e = \sum_{j=1}^a (\bv_{\wX_j})_e
\]
 for $e=(i,m)$ or $e=(m,i)$.
Since all the $D_j$ are connected, if $e=(m,\ell)$ with $m,\ell\neq i$ then $(\bv_{\wX_j})_e=0$ for all $j$. 
The claim follows.
We conclude that  $\mathrm{rowsp}(A)= \mathrm{span}\{\textbf{v}_{\wX} \mid \wX\in\cC^*\}$ and thus $\dim(P_{\cM^*(\widetilde{G})})=\mathrm{rank}(A)= |V|-k$.
\end{proof}

Together, Proposition \ref{thm_graph_dim} and Proposition \ref{thm_dualgraph_dim} give us that
\begin{align*}
	\dim(P_{\mathcal{M}(\widetilde{G})}) + \dim(P_{\mathcal{M}^*(\widetilde{G})}) = |E|,
\end{align*}
that is, given a graphical matroid, the sum of the dimensions of the OMC polytope of the graphical matroid and the OMC polytope of its dual is equal to the ambient dimension. However, as we can see in the following example this does not hold for any oriented matroid. 

\begin{example}\label{eg:typeB}
Consider the matrix
\begin{align*}
	M=\begin{pmatrix}
		1 & 0 & 0 & 1 &  1 & 1 &  1 & 0 &  0 \\
		0 & 1 & 0 & 1 & -1 & 0 &  0 & 1 &  1 \\
		0 & 0 & 1 & 0 &  0 & 1 & -1 & 1 & -1 
	\end{pmatrix}
\end{align*}
whose columns are the type $B_3$ positive roots.
The oriented matroid
$\mathcal{M}(\mathcal{B}_3^+)$ corresponding to the matrix $M$ is not a graphical matroid, and $\dim(P_{\mathcal{M}(\mathcal{B}_3^+)}) = 9$, which is the dimension of the ambient space. Consider, then, the dual oriented matroid. A direct computation shows that $P_{\mathcal{M}^*(\mathcal{B}_3^+)}$ is not a $0$-dimensional polytope; in fact, using a computer, we determined that $\dim(P_{\mathcal{M}^*(\mathcal{B}_3^+)}) = 9$ as well.
\end{example}


\subsection{Bridges and loops}

In \Cref{sec_dual_a}, we carefully study the oriented matroid cocircuit polytope whose underlying matroid is the complete graph on $n$ nodes, $K_n$. We do not study other OMC polytopes in such detail, but we can make some slightly more general comments. 

\begin{proposition}\label{prop:suspension}
	Let $\widetilde{G}=(V,E)$ be a directed graph with at least one bridge $b\in E$. Then $P_{\cM^*(\widetilde{G})}$ is a suspension of $P_{\cM^*(\widetilde{G}\setminus\{b\})}$.
\end{proposition}

\begin{proof}
	Since a bridge is an edge whose removal increases the number of connected components by one, $b$ precisely defines a minimal cut of $\widetilde{G}$. Thus $b$ yields cocircuits of the form $\widetilde{X}=(0,\dots,0,+,0,\dots,0)$ and $-\widetilde{X}=(0,\dots,0,-,0,\dots,0)$. By circuit axiom (C2), there can be no other circuits of $\cM^*(\widetilde{G})$ containing $\widetilde{X}$, so every other circuit of $\cM^*(\widetilde{G})$ comes from a circuit of $\cM^*(\widetilde{G}\setminus\{b\})$, with a $0$ in the spot corresponding to edge $b$. In particular, $P_{\cM^*(\widetilde{G})}$ is a suspension of $P_{\cM^*(\widetilde{G}\setminus\{b\})}$.
\end{proof}

Consider another special kind of edge: a directed loop. For example, let $\widetilde{G}$ be the directed graph that is a set of three directed loops on one node. Every loop contributes two circuits: one from traversing the loop in a positive direction, and the other from traversing the loop in a negative direction. The circuits, then, are $(1,\emptyset)$, $(\emptyset, 1)$, $(2,\emptyset)$, $(\emptyset, 2)$, $(3,\emptyset)$, $(\emptyset, 3)$. Then $P_{\cM(\widetilde{G})}$ has vertices $(1,0,0)$, $(-1,0,0)$, $(0,1,0)$, $(0,-1,0)$, $(0,0,1)$, $(0,0,-1)$: it is the three-dimensional cross-polytope. Extending this idea gives us the following proposition.

\begin{proposition}\label{prop:bouquet}
	Let $\widetilde{B}=(V,E)$ be a bouquet, that is, a graph with $|E|$ directed loops. Then $P_{\cM(\widetilde{B})}$ is the $|E|$-dimensional cross-polytope.
\end{proposition}


\section{A closer look: The complete graph}\label{sec_dual_a}

Given the complete graph on $n$ nodes $K_n$, orient the edges $(i,j)$ with $i<j$.
Observe that the incidence matrix of this directed graph is the matrix whose columns are the positive roots of $A_{n-1}$, that is, $\be_i-\be_j$ with $1\le i<j\le n$.
Since the oriented matroid for this directed graph is represented by this matrix, we denote this matroid by $\cA_{n-1}$ and its dual by $\cA_{n-1}^*$.
In this section we focus on the polytope $P_{\cA_{n-1}^*}\subseteq \RR^{\binom{n}{2}}$, which we denote by $\Poly_{n-1}$ to simplify notation.
The polytope $\Poly_{n-1}$ appears in several other contexts, which we use to determine the structure, Ehrhart theory, and even equivariant Ehrhart theory of $\Poly_{n-1}$ with respect to the previously mentioned $S_n$ action. 
We also describe this polytope and the fixed polytopes under this action.

\subsection{The polytope $\Poly_{n-1}$}\label{sec_poly_details}

Let $\widetilde{K_n}$ be the complete graph on $n$ nodes, oriented so that each edge is directed from the smaller vertex to the larger vertex. Note that by \Cref{rem:reverseorientation}, reversing the orientations of some of the edges would yield an isomorphic polytope. By \Cref{vertices}, we have that each vertex of $\Poly_{n-1}$ corresponds with a minimal cut $[n]=V_1\sqcup V_2$ of $K_n$.
In fact, any partition $[n]=V_1\sqcup V_2$ with $V_1,V_2\neq\emptyset$ is a minimal cut of $K_n$.
Throughout this section we denote by $\widehat{\textbf{u}}_{a_1\dots a_m}$ the vertex of $\Poly_{n-1}$ corresponding to the minimal cut with $V_1=\{a_1,\dots,a_m\}$. In particular, the vertices $\widehat{\bu}_{I}$ of $\Poly_{n-1}$ are in bijection with subsets $\emptyset\subsetneq I\subsetneq [n]$.

The following corollary is an immediate consequence of Proposition \ref{thm_dualgraph_dim}.

\begin{corollary}\label{dAdim}
	The dimension of $\Poly_{n-1}$ is $n-1$.
\end{corollary}

\begin{figure}
	\centering
	\includegraphics[width=0.3\textwidth]{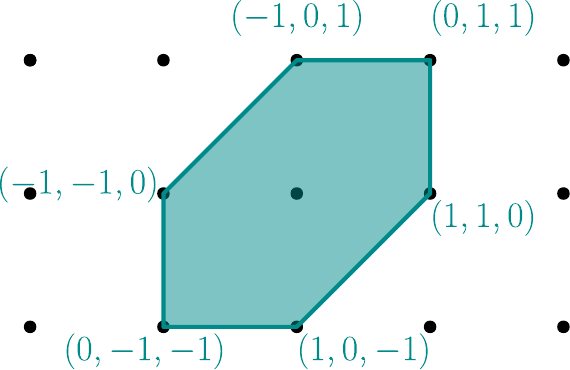}
	\caption{$\Poly_{2}$ with aff$(\Poly_{2})$ pictured as generated by $(1,1,0)$ and $(-1,0,1)$.}
	\label{2withcoords}
\end{figure}

We write the coordinates of $\RR^{\binom{n}{2}}$ as $x_{ij}$ for $1\leq i<j\leq n$, where $x_{ij}$ corresponds to the edge $(i,j)$ of $K_n$.
Note then that given $V_1=\{a_1,\dots,a_m\}$,
\begin{align*} (\widehat{\textbf{u}}_{a_1\dots a_m})_{ij} := 
	\begin{cases}
		1, & i\in V_1,\ j\in V_2 \\
		-1, & i\in V_2,\ j\in V_1 \\
    0, & \text{otherwise}
	\end{cases}
.\end{align*}
It is straightforward to check that given $1\leq i<j< n$, every vertex of $\Poly_{n-1}$ satisfies   $x_{ij}-x_{in}+x_{jn} = 0$.
Thus, $\Poly_{n-1}$ is contained in the linear subspace of $\RR^{\binom{n}{2}}$ given by the intersection of the $\binom{n-1}{2}$ hyperplanes of the form $x_{ij}-x_{in}+x_{jn} = 0$. These hyperplanes correspond to the cycles of $K_n$ of the form $ijn$. The reader can confirm that this intersection is $(n-1)$-dimensional. Thus, it follows from \Cref{dAdim} that this linear subspace is the affine span $\mathrm{aff}(\Poly_{n-1})$ of $\Poly_{n-1}$.

In order to better understand the structure of $\Poly_{n-1}$, we discuss zonotopes. For a finite set of vectors $\bV=\{\bv_1,\ldots,\bv_n\} \subseteq \RR^d$ we define the \textbf{zonotope} of $\bV$ to be
\[
\Zono(\bV) \ := \  \sum_{i=1}^n [\bzero, \bv_i] \ := \ \{\lambda_1\bv_1+\dots+\lambda_n\bv_n \mid \text{for all $i$, } 0\leq\lambda_i\leq 1\} \, ,
\]
where $[\bzero, \bv_i]:=\{\lambda\bv_i \mid 0\le \lambda\le 1\}$.

\begin{proposition}\label{itsazonotope}
	$\Poly_{n-1}$ is the zonotope $\Zono(\widehat{\bu}_1,\widehat{\bu}_2,\dots,\widehat{\bu}_n)$.
\end{proposition}

\begin{proof}
First, note that \eqref{eq_bond_expansion} implies that $\widehat{\bu}_{a_1\dots a_m} = \sum\limits_{i=1}^{m} \widehat{\bu}_{a_i}$ for each $\{a_1,\dots,a_m\}\neq \varnothing,[n]$. 
Thus, each vertex of $\Poly_{n-1}$ is in $\Zono(\widehat{\bu}_1,\widehat{\bu}_2,\dots,\widehat{\bu}_n)$, i.e.\ $\Poly_{n-1} \subseteq \Zono(\widehat{\bu}_1,\widehat{\bu}_2,\dots,\widehat{\bu}_n)$.
To prove the opposite containment it suffices to verify that each sum $\lambda_1\widehat{\bu}_1+\dots+\lambda_n\widehat{\bu}_n$ with $\lambda_i\in\{0,1\}$ is in $\Poly_{n-1}$.
The case in which $\lambda_i=0$ and $\lambda_j=1$ for some $i,j$ follows from the equation above.
Since $\bzero\in\Poly_{n-1}$, the case in which all $\lambda_i=0$ also follows.
For the case in which all $\lambda_i=1$, note that for any edge $e$
    $$
    (\widehat{\bu}_1)_e+\dots+(\widehat{\bu}_n)_e=1-1=0
    $$ 
and therefore $\widehat{\bu}_1+\dots+\widehat{\bu}_n=\bzero\in\Poly_{n-1}$.
\end{proof}

\begin{remark}
One could ask if all graphical matroids yield OMC polytopes that are zonotopes, but it is not the case. Consider, for instance, the graph $G$ which is $\widetilde{K_4}$ with the edge $\overrightarrow{14}$ removed. One can observe that the corresponding cocircuit polytope $P_{\cM^*(G)}$ is a 3-dimensional polytope with 14 faces, 8 of which are triangles. Since the faces of a zonotope are also zonotopes but a triangle is not a zonotope, $P_{\cM^*(G)}$ is not a zonotope. 
\end{remark}

It is often easier to work with full-dimensional objects; since $\Poly_{n-1}$ is an $(n-1)$-dimensional polytope, we want to view $\Poly_{n-1}$ as a polytope in $\RR^{n-1}$. We use the projection in which we keep the $x_{in}$ coordinates for $i\in[n-1]$. Since each coordinate is of the form $x_{in}$ for a fixed $n$, we now write the coordinates as $x_i$ for $i\in[n-1]$. In particular, the projection $\pi:\RR^{\binom{n}{2}}\to\RR^{n-1}$ comes from the $(n-1)\times\binom{n}{2}$ matrix
\begin{align*}
	(\pi_{ij}) =
	\begin{cases}
		1 &\text{ if } j = in-\binom{i+1}{2} \\
		0 &\text{ otherwise}
	\end{cases}
\end{align*}
and $\pi(\Poly_{n-1})=\Zono(\be_1,\ldots,\be_{n-1},-\mathbb{1})$, where $\{\be_1,\ldots,\be_{n-1}\}$ is the standard basis for $\RR^{n-1}$ and $\mathbb{1}:=\be_1+\cdots+\be_{n-1}$. For $\emptyset\subsetneq I\subsetneq[n]$, we define $\bu_I := \pi(\widehat{\bu}_I)$.

Not only is $\pi(\Poly_{n-1})$ linearly isomorphic to $\Poly_{n-1}$, but the two polytopes have the same lattice point count. 
To verify this, first observe that if $p\in\ZZ^{\binom{n}{2}}$, it immediately follows that $\pi(p)\in\pi\left(\ZZ^{\binom{n}{2}}\right)$. On the other hand, consider $p\in\pi\left(\ZZ^{\binom{n}{2}}\right)$, so $p = (a_1,a_2,\dots,a_{n-1})$ where $a_i\in\ZZ$. To construct $\pi^{-1}(p) = (a_{12}, a_{13}, \dots, a_{n-1 n})$, we must have $a_{in} = a_i$. We can then use the equations of the hyperplanes $x_{ij} - x_{in} + x_{jn} = 0$ to ensure that we end up with a point in $\ZZ^{\binom{n}{2}}\cap\mathrm{aff}(\Poly_{n-1})$. This tells us that $a_{ij} = a_{in} - a_{jn} = a_i - a_j$. Then $a_i,a_j\in\ZZ$ implies that $a_{ij}\in\ZZ$ as well. Thus $\pi^{-1}(p)\in\ZZ^{\binom{n}{2}}\cap\mathrm{aff}(\Poly_{n-1})$.

\begin{figure}
	\centering
	\includegraphics[width=0.25\textwidth]{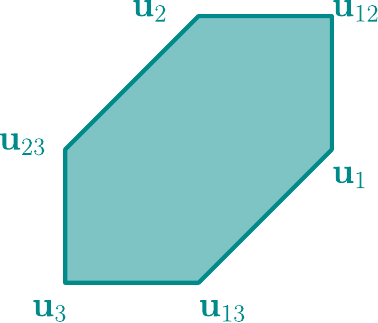}
	\hspace{24pt}
	\includegraphics[width=0.32\textwidth]{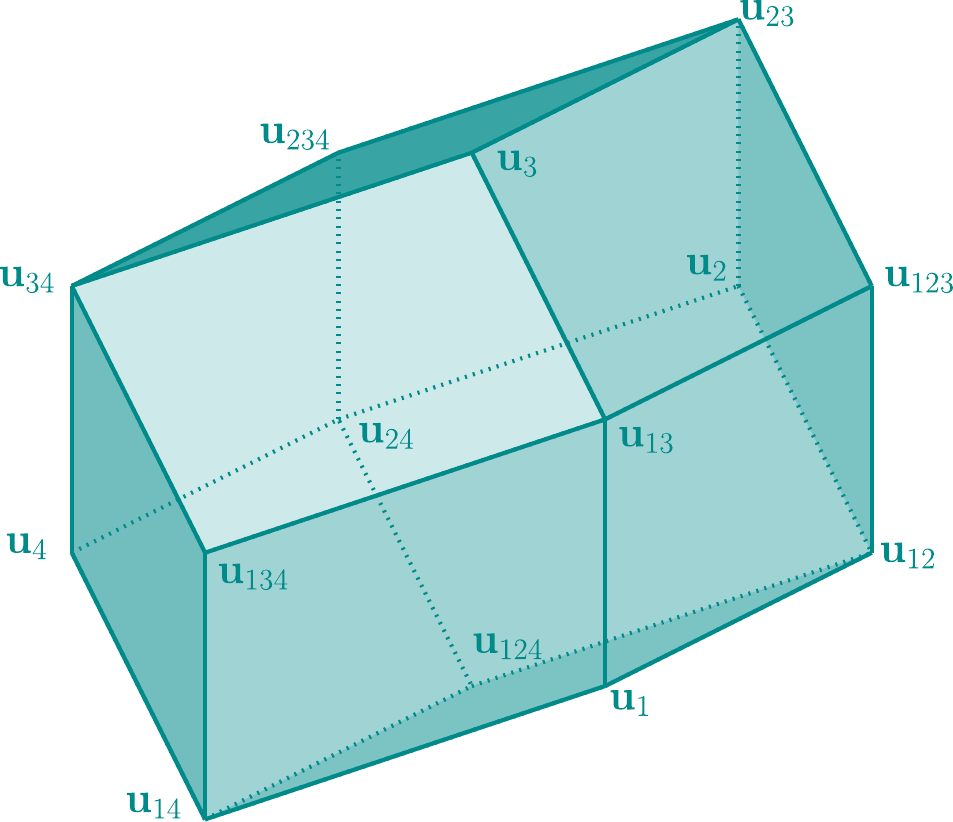}
	\caption{$\Poly_{2}$ (left) and $\Poly_{3}$ (right).}
	\label{2and3labeled}
\end{figure}

\subsection{Other appearances} \label{sec_other_appearances}

In addition to being an OMC polytope, $\Poly_{n-1}$ is found in other, well-studied families of polytopes, in particular, graphic zonotopes and the polar duals of symmetric edge polytopes. Moreover, $\Poly_{n-1}$ appears in Stapledon's paper \cite{stapledonequivariant} in Example 8.7 and Remark 9.4; we shall return to this last appearance in more detail in \Cref{sec_eet}.

First, we verify that $\Poly_{n-1}$ is affinely isomorphic to the graphic zonotope of the cycle with $n$ vertices, denoted $C_n$. 
The \textbf{graphic zonotope} $\Zono_G$ associated to a graph $G=([n],E)$ is 
\begin{equation*}
	\Zono_G:=\sum_{\{u,v\}\in E}[\be_u,\be_v ].
\end{equation*}

\begin{proposition}\label{prop: zono is graph zono}
Consider the affine map
$\phi:\RR^{n-1}\to\RR^{n}$ given by
\begin{equation*}
	\phi(\bx):=(-x_1+1,x_1-x_2+1,\ldots,x_{n-2}-x_{n-1}+1,x_{n-1}+1)
.\end{equation*}
 We have that $\phi(\Zono(\be_1,\ldots,\be_{n-1},-\mathbb{1}))=\Zono_{C_n}$ and thus $\Poly_{n-1}$ is affinely isomorphic to $\Zono_{C_n}$.
\end{proposition}

\begin{proof}
Fix $\bv=\lambda_1\be_1+\cdots+\lambda_{n-1}\be_{n-1}-\lambda_n\mathbb{1}=(\lambda_1-\lambda_n,\cdots,\lambda_{n-1}-\lambda_n)\in\Zono(\be_1,\ldots,\be_{n-1},-\mathbb{1})$ so that $\lambda_i\in[0,1]$ for all $i$. Direct computation shows that
\begin{align*}
	\phi(\bv)
	&=
	(1-\lambda_1+\lambda_n,\lambda_1+1-\lambda_2,\cdots,\lambda_{n-2}+1-\lambda_{n-1},\lambda_{n-1}+1-\lambda_n)
	\\&=
	((1-\lambda_1)\be_1+\lambda_1\be_2)+\cdots+((1-\lambda_{n-1})\be_{n-1}+\lambda_{n-1}\be_n)+((1-\lambda_n)\be_n+\lambda_n\be_1) \stepcounter{equation}\tag{\theequation}\label{eq:graphzonophi}
\end{align*}
and the last expression shows that $\phi(\bv)\in\Zono_{C_n}$.
The opposite containment also follows from \Cref{eq:graphzonophi}.
\end{proof}

Given a graph $G$ with node set $V = [n]$, the \textbf{symmetric edge polytope} (SEP) associated to $G$ is the polytope
\begin{align*}
	\mc{S}_G := \conv\{\pm(\be_i-\be_j)\mid \{i,j\}\in E(G)\}.
\end{align*}
Symmetric edge polytopes themselves are quite interesting (see, for example, \cite[Section 6.1]{Ferroni2024})
and sometimes make an appearance in other contexts (see, for instance, \cite[Section 2]{manyfacessep}). When $G$ is the complete graph $K_n$ on $n$ nodes, $\mc{S}_{K_n}$ is the root polytope of $A_{n-1}$. 

The polar duals of symmetric edge polytopes are also of interest \cite{sepgentoregmats}.
    Let $P$ be a $d$-dimensional lattice polytope containing the origin in its interior. The \textbf{polar dual} of $P$ is the polytope
\begin{align*}
	P^{\vee} := \{\bu\in\RR^d \mid \bu\cdot\bx\leq 1 \text{ for every } \bx\in P\}.
\end{align*}
When $P^{\vee}$ is a lattice polytope, we say that $P$ is \textbf{reflexive}.
Symmetric edge polytopes are known to be reflexive.
We find a connection to OMC polytopes: 
the polar dual of the SEP $\mc{S}_{K_n}$ is the tropical unit ball, defined in {\cite[Section 2]{criadotropicalbisectors}} and further investigation shows that, in fact, $\Poly_{n-1}$ is a tropical unit ball.
Concretely, $\Poly_{n-1} = \mc{S}_{K_n}^{\vee}$. 
As a consequence we obtain the following.

\begin{proposition}
The polytope $\Poly_{n-1}$ is reflexive.
\end{proposition}

One could ask if the relation between SEPs and OMC polytopes holds for all graphs, but this is not the case.
Consider, for example, the graph $G = K_4\setminus e$, where $e$ is the edge $24$, and let $\widetilde{G}$ be the directed graph obtained by orienting the edges of $G$ as $\overrightarrow{12}$, $\overrightarrow{23}$, $\overrightarrow{34}$, $\overrightarrow{41}$, and $\overrightarrow{13}$. The polar dual of the SEP of $\widetilde{G}$ is given in \cite[Example 5.2]{sepgentoregmats}. The polytopes $\mc{S}_{\widetilde{G}}^\vee$ and $P_{\cM^*(\widetilde{G})}$, pictured in \Cref{SEP_duals}, are distinct polytopes.
\begin{figure}[h!]
	\centering
	\includegraphics[width=0.3\textwidth]{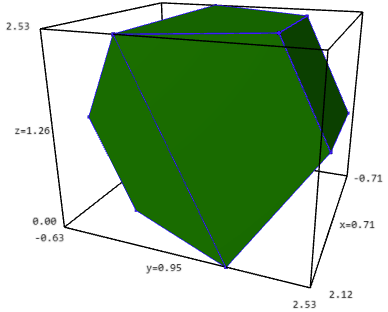}
	\hspace{24pt}
	\includegraphics[width=0.32\textwidth]{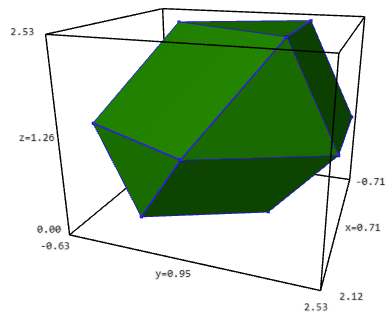}
	\caption{The polar dual of the symmetric edge polytope $\mc{S}_{\widetilde{G}}$ (left) and the OMC polytope $P_{\cM^*(\widetilde{G})}$ (right).}
	\label{SEP_duals}
\end{figure}
In particular, one can see that the facets of $\mc{S}_{\widetilde{G}}^\vee$ are all quadrilaterals, whereas the facets of $P_{\cM^*(\widetilde{G})}$ are both quadrilaterals and triangles.

\subsection{Face structure of $\Poly_{n-1}$} \label{sec_face_structure}

Next, we study the face structure of $\Poly_{n-1}$. 
In \cite[Lemma 13.4]{AguiarArdila}, a concrete bijection is given between the proper faces of a graphic zonotope $\Zono_G$ and the pairs $(f,\mc{O})$, where $f\subseteq E$ is a flat of the graph $G$ and $\mc{O}$ is an acyclic orientation of the contraction $G/f$. 
To construct the face corresponding to $(f,\mc{O})$, one first constructs the graph $g(f,\mc{O})$ obtained from $G$ by fixing the edges in $f$ and replacing each edge $\{v,w\}\notin f$ with the half-edge $\{v\}$, where $v\to w$ in $\mc{O}$.
The face corresponding to $(f,\mc{O})$ is the zonotope
\begin{align*}
    \Zono_{g(f,\mc{O})} &= \sum_{ \stackrel{\{v_j\} \text{is a half-edge}}{ \text{ \rm of } g(f,\mc{O}) }} \Delta_{j} + \sum_{ \stackrel{e_j \text{is an edge}}{ \text{ \rm of } g(f,\mc{O}) }} \Delta_{j,j+1},
\end{align*}
where $\Delta_{j} = \be_j$ and $\Delta_{j,j+1} = [\be_j,\be_{j+1}]$.
In the following theorem, we use this description in the case $G=C_n$ to describe the faces of $\Poly_{n-1}$.

\begin{theorem}\label{thm_faces}
    For $0\le i<n-1$, the following is a bijection
    \begin{align*}
    \{(S,T)\, : \, \varnothing\subsetneq S \subseteq T\subsetneq [n],\ |T\setminus S|=i\}
    &\to\{i\text{-dimensional faces of } \Poly_{n-1}\}\\
    (S,T)
    &\mapsto \sum_{j\in S}\bu_j+\sum_{j\in T\setminus S}[\bzero,\bu_j]
    .\end{align*}
\end{theorem}

\begin{proof} 
    To show the desired bijection, we use the fact that $\Poly_{n-1}$ is affinely isomorphic to the graphic zonotope $\Zono_{C_n}$. 
    We first identify a bijection between the set of pairs $(S,T)$ and the set of pairs $(f,\mc{O})$. Then, we compose this bijection with the map $\phi$ in \Cref{prop: zono is graph zono}.
	
    Start with a planar realization of $C_n$. Throughout this proof, our labels are integers modulo $n$; in particular, $n+1 = 1$. Denote by $v_1,\dots,v_n$ the vertices of $C_n$, oriented so that the labels increase clockwise, and let $e_j\in E$ denote the edge $\{v_j, v_{j+1}\}$. Let $f_I = \{e_j\mid j\in I\}$. Then the flats of $C_n$ are $f_I$ for $I = [n]$ and $I\in\binom{[n]}{k}$ for $k=0,\dots,n-2$. Since we are interested in proper faces and the face of the polytope corresponding to $I=[n]$ is the polytope itself, we exclude this case.
	
    Consider the pair $(S,T)$ and the flat $f = f_{T\setminus S}$. We construct the following orientation $\mc{O}$ from $S$, assigning each edge $e_j\in C_n/f$ the orientation $\mc{O}_j$:
    \begin{align} \label{eqn:orientation}
		\mc{O}_j = 
		\begin{cases}
		  \text{counterclockwise}, &j\in S \\
            \text{clockwise}, &j\notin S.
		\end{cases}   
    \end{align}

    This map is surjective: consider a flat and orientation pair $(f,\mc{O})$. Then 
    \begin{align*}
		S &= \{j\in[n]\mid e_j \text{ has a counterclockwise orientation in } \mc{O}\} \\
		T &= S\cup\{j\in[n]\mid e_j\in f\}.
    \end{align*}
    Observe that $S\neq\emptyset$ since $S=\emptyset$ would yield an all-clockwise orientation and that $T\neq[n]$ since $T=[n]$ would yield an all-counterclockwise orientation, neither of which are acyclic.
    	
    To see that we have a bijection, we compare the cardinality of both sets.
	
    For $\{(S,T)\mid \emptyset\subsetneq S\subseteq T\subsetneq [n]\}$, since $S\neq\emptyset$ and $T\neq [n]$, $T\setminus S$ can have size $k=0,\dots,n-2$, and for each value of $k$, there are $\binom{n}{k}$ possible choices for elements of $T\setminus S$. Now, given $T\setminus S$ with $|T\setminus S| = k$, we examine the remaining $n-k$ elements of $[n]$. Each element can either be in $S$ or not, giving $2^{n-k}$ total options. However, since $S\neq\emptyset$ and $T\neq[n]$, we subtract two of those options. Thus the cardinality is $\sum_{k=0}^{n-2}\binom{n}{k}(2^{n-k}-2)$.
	
    For $\{(f,\mathcal{O})\mid f\neq E \text{ flat of } C_n, ~ \mc{O} \text{ acyclic orientation of } C_n/f\}$, we obtain a flat $f_I$ for each $I\in\binom{[n]}{k}$ for $k=0,\dots,n-2$. For each flat $f_I$ with $|I|=k$, we need to count the acyclic orientations of $C_n/f_I$. $C_n/f_I$ is now another cycle graph, with $n-k$ edges. The only orientations that are not acyclic are the all-clockwise and all-counterclockwise orientations. Thus overall we obtain $2^{n-k}-2$ acyclic orientations. Putting this all together, the cardinality is $\sum_{k=0}^{n-2}\binom{n}{k}(2^{n-k}-2)$, as desired.
	
    Thus we have a bijection 
     \begin{align*}
         \{(S,T)\mid \emptyset\subsetneq S\subseteq T\subsetneq [n]\} \leftrightarrow \{(f,\mathcal{O})\mid f\neq E \text{ flat of } C_n, ~ \mc{O} \text{ acyclic orientation of } C_n/f\}.
     \end{align*}
    We next want to relate the elements $(f,\mc{O})$ with proper faces of $\Zono_{C_n}$. 
    Applying \cite[Lemma 13.4]{AguiarArdila} and the above bijection, we get 
    \begin{align*}
		\Zono_{g(f_{T\setminus S},\mc{O})} &= \sum_{ \stackrel{\{v_j\} \text{is a half-edge}}{ \text{ \rm of } g(f_{T\setminus S},\mc{O}) }} \Delta_{j} + \sum_{ \stackrel{e_j \text{is an edge}}{ \text{ \rm of } g(f_{T\setminus S},\mc{O}) }} \Delta_{j,j+1} \\
		&= \sum_{j\in T\setminus S} \Delta_{j,j+1} + \sum_{j\in S} \Delta_{j+1} + \sum_{j\notin T} \Delta_{j},
    \end{align*}
    where $\Delta_{j} = \be_j$ and $\Delta_{j,j+1} = [\be_j,\be_{j+1}]$. Our final step, then, is to determine which faces of $\Poly_{n-1}$ correspond to each $\Zono_{g(f,\mc{O})}$. Towards this step, consider the map $\phi: \Poly_{n-1}\to\Zono_{C_n}$ given in \Cref{prop: zono is graph zono}. We claim that 
    \begin{align*}
		\sum_{j\in S} \bu_j + \sum_{j\in T\setminus S} [\bzero,\bu_j]\mapsto \Zono_{g(f_{T\setminus S},\mc{O})},
    \end{align*}
    where $\mc{O}$ is the orientation defined in (\ref{eqn:orientation}). 
	
    Let $\bv\in \sum\limits_{j\in S} \bu_j + \sum\limits_{j\in T\setminus S} [0,\bu_j]$. To determine $\phi(\bv)$, we rewrite $\bv$ as a single sum:
    \begin{align*}
		\bv &= \sum_{j\in S} \bu_j + \sum_{j\in T\setminus S} \lambda_j \bu_j,  &&\lambda_j\in[0,1] ~ \text{ for } j\in T\setminus S \\
		&= \sum_{j\in[n]} \lambda_j  \bu_j, &&\lambda_j = 
		\begin{cases}
		  1, &j\in S \\
	   	0, &j \notin T \\
		  x \in[0,1], &j\in T\setminus S.
		\end{cases}
    \end{align*}
    Then, by the proof of \Cref{prop: zono is graph zono},
    \begin{align*}
		\phi(\bv) &= ((1-\lambda_1)\be_1 + \lambda_1\be_2) + \dots + ((1-\lambda_{n-1})\be_{n-1} + \lambda_{n-1}\be_n) + ((1-\lambda_n)\be_n + \lambda_n\be_1).
    \end{align*}
    Now, we can consider each of the three cases:
    \begin{enumerate}
		\item If $j\in S$, then $\lambda_j = 1$, so $(1-\lambda_{j})\be_{j} + \lambda_{j}\be_{j+1} = \be_{j+1} = \Delta_{j+1}$.
		\item If $j\notin T$, then $\lambda_j = 0$, so $(1-\lambda_{j})\be_{j} + \lambda_{j}\be_{j+1} = \be_{j} = \Delta_{j}$.
		\item If $j\in T\setminus S$, then $\lambda_j\in[0,1]$, so $(1-\lambda_{j})\be_{j} + \lambda_{j}\be_{j+1} \in \Delta_{j,j+1}$.
    \end{enumerate}
    Thus $\phi(\bv)\in\Zono_{g(f_{T\setminus S},\mc{O})}$ as desired. Moreover, varying $\lambda_j$ in the third case yields the whole interval $\Delta_{j,j+1}$, and thus this map is surjective.
\end{proof}

The following follows from the description of the face corresponding to $(S,T)$ in the preceding theorem.

\begin{corollary}\label{cor_face_lattice}
    The face lattice of $\Poly_{n-1}$ is isomorphic to the poset $(\{(S,T)\, : \, \varnothing\subsetneq S \subseteq T\subsetneq [n]\},\preceq)$, where $(S_1,T_1)\preceq (S_2,T_2)$ if and only if $T_1\subseteq T_2$ and $S_2\subseteq S_1$.
\end{corollary}

Given a $d$-dimensional polytope $\cP$ and $0\le i\le d$, let $f_i(\cP)$ denote the number of $i$-dimensional faces of $\cP$.
The $f$-polynomial of $\cP$ is then $f_\cP(t):=\sum_{i=0}^df_i(\cP)\,t^i$. Using the theorem above, we immediately obtain the $f$-polynomial of $\Poly_{n-1}$, which due to \Cref{prop: zono is graph zono} recovers a formula by Gruji\'{c}.

\begin{corollary}\cite[Proposition 5.2]{Gru}
    The $f$-polynomial of $\Poly_{n-1}$ is $f_{\Poly_{n-1}}(t)=t^{n-1}+\sum_{i=0}^{n-2}(2^{n-i}-2)\binom{n}{i}t^i$.
\end{corollary}

\subsection{Ehrhart theory of $\Poly_{n-1}$} \label{sec_ehrhart}

We now introduce Ehrhart theory. The \textbf{lattice point enumerator} of a polytope $P$ is
\begin{align*}
    L_{P}(t) := \#\left(tP \cap \ZZ^d\right).
\end{align*}
\begin{theorem}\cite{ehrhartpolynomial}\label{ehrpoly}
    Let $P$ be an integral convex $d$-polytope; then $L_{P}(t)$ is a polynomial in $t$ of degree $d$ with leading term $\vol(P)$ and constant term 1.
\end{theorem}

When $P$ is integral, we call $L_P(t)$ the \textbf{Ehrhart polynomial} of $P$ and sometimes denote it $\ehr_P(t)$ instead of $L_P(t)$. 
The generating function of $L_P(t)$ is called the \textbf{Ehrhart series} of $L_\Poly(t)$ and is denoted $\Ehr_P(z)$: 
	\begin{equation*} 
    	\Ehr_P(z) := 1 + \sum_{t\geq1}L_P(t)z^t  = \frac{h_d^*z^d + h_{d-1}^*z^{d-1}+\dots+h_1^*z + h_0^*}{(1-z)^{d+1}}.
	\end{equation*}
	The numerator polynomial $h_P^*(z)$ is called the \textbf{$h^*$-polynomial} of $P$, and when written in vector form, it is called the \textbf{$h^*$-vector}: $h^*(P) := (h^*_0, \dots, h^*_d)$.

Next, we describe the Ehrhart polynomial and series of $\Poly_{n-1}$.
Since Eulerian polynomials will play a role, we proceed to introduce them.
For a positive integer $n$, the \textbf{Eulerian polynomial} is $A_n(t)=\sum_{i=0}^{n-1}A(n,i)\,t^i$,
where the \textbf{Eulerian number} $A(n,i)$ is the number of permutations in~${S}_n$ with exactly $i$ descents. The Ehrhart polynomial of the dual polytope to the convex hull of the root system
of type $A$ was computed in \cite[Lemma 5.3]{higashitanikummermichalek}.
The following proposition computes the Ehrhart theory of $\Poly_{n-1}$ as a simple corollary of \cite[Exercise 4.64b]{stanleyec1}.

\begin{proposition}
	The Ehrhart polynomial of $\Poly_{n-1}$ is
	\begin{align*}
		\ehr_{\Poly_{n-1}}(t) = (t+1)^n - t^n = \sum_{k=0}^{n-1}\binom{n}{k}t^k.
	\end{align*}
	The Ehrhart series of $\Poly_{n-1}$ is
	\begin{align*}
		\Ehr_{\Poly_{n-1}}(z) = \frac{A_n(z)}{(1-z)^n}.
	\end{align*}
    \label{ehrdA}
\end{proposition}

\begin{proof}
    We apply Stanley's method for computing the Ehrhart polynomial of graphic zonotopes \cite[Exercise 4.64b]{stanleyec1}. Since $\Poly_{n-1}\cong \Zono_{C_n}$, we need to count the number of subforests of $C_n$ with $k$ edges to obtain the coefficient of $t^k$ in $\ehr_{\Poly_{n-1}}(t)$. Since $C_n$ is itself just a cycle, any subset of $k\neq n$ edges forms a forest. In particular, the coefficient of $t^k$ is $\binom{n}{k}$ for $0\leq k\leq n-1$.

    Then the Ehrhart series is
	\begin{align*}
		\Ehr_{\Poly_{n-1}}(z) 
		&
		= 1+\sum_{t\geq 1}\left((t+1)^n - t^n\right)z^t 
        \\&
		= \sum_{t\geq 0}(t+1)^n z^t - \sum_{t\geq 1} t^n z^t 
		\\ 	&
		= \sum_{t\geq 1} \frac{t^n}{z} z^t - \sum_{t\geq 1} t^n z^t 
		\\	&
		= \left(\frac{1-z}{z}\right)\sum_{t\geq 1}t^n z^t 
        \\
		&= (1-z)\frac{\sum_{k=1}^{n} A(n,k)z^{k-1}}{(1-z)^{n+1}} 
        \\
		&= \frac{A_n(z)}{(1-z)^n},
	\end{align*}
	where the last equation follows from a well known identity of the Eulerian polynomial, see e.g.\ \cite[Proposition 1.4.4]{stanleyec1}.\footnote{We warn the reader that the conventions in \cite{stanleyec1} for Eulerian polynomials differ from ours by a factor of $t$.}
\end{proof}

We remark that the $h^*$-polynomial of $\Poly_{n-1}$ is strikingly similar to the $h^*$-polynomial of the $n$-dimensional unit cube $\square_{n}$, which is
\begin{equation*}
    \Ehr_{\square_n}(z) = \frac{A_n(z)}{(1-z)^{n+1}}.
\end{equation*}

The $h^*$-polynomial appearing in the previous Proposition has log-concave coefficients. We ask whether this phenomenon holds more generally for oriented matroid circuit polytopes.
\begin{conjecture}\label{conj}
Let $\mathcal{M}$ be an oriented matroid, and let $P_{\mathcal{M}}$ be its oriented matroid circuit polytope. The $h^*$-vector of $P_{\mathcal{M}}$ is log-concave.
\end{conjecture}

We have verified Conjecture~\ref{conj} computationally for all oriented matroids in Finschi's catalog of isomorphism classes of oriented matroids with cardinality at most $7$~\cite{FinschiCatalogOM} (see also Finschi and Fukuda~\cite{FinschiFukuda2002} for the underlying
generation methods). In each case, the $h^*$-vector of the associated OMC polytope was log-concave.


\subsection{Equivariant Ehrhart theory of $\Poly_{n-1}$} 
\label{sec_eet}

We begin by introducing equivariant Ehrhart theory for general polytopes. The polytopes we are examining are all full-dimensional, and thus their affine spans contain the origin, so we will focus on that specific case; for the more general construction, see Stapledon~\cite{stapledonequivariant}.
Let $P\subset\RR^n$ be a $n$-dimensional lattice polytope and $G$ be a group that acts linearly on $\ZZ^n$ and such that $P$ is invariant under the action of $G$.  Let $\sf{M}$ be the intersection of $\ZZ^n$ with the linear span of $P$. Now $\sf{M}$ is $G$-invariant, and we obtain a representation $\rho:G\to GL(\sf{M})$. 

Let $\chi_{tP}$ denote the permutation character associated to the action of $G$ on the lattice points in the $t$th dilate of $P$. Given $g\in G$, let 
\begin{align*}
	P^g = \{p\in P\mid g\cdot p = p\},
\end{align*}
that is, $P^g$ is the polytope of points $p\in P$ fixed by $g$. In \cite[Lemma 5.2]{stapledonequivariant} it is shown that $\chi_{tP}(g) = L_{P^g}(t)$.

Consider the function $H^*:G\to\ZZ[[z]]$ given by
\begin{align*}
	H^*[z](g) := (1-z)\det(I-g\cdot z)\sum_{t\geq 0}\chi_{tP}(g)z^t,
\end{align*}
called the \textbf{equivariant $H^*$-series} of $P$. If we rearrange the above equation to isolate the sum, we obtain something that looks like an Ehrhart series; indeed, when we consider the identity element $e\in G$, $H^*[z](e)$ is in fact the $h^*$-polynomial of~$P$.

We now describe the linear $S_n$-action on $\Poly_{n-1}$ with respect to which we will compute the equivariant Ehrhart theory. 
Stapledon \cite{stapledonequivariant} uses a connection with the cohomology of the toric variety of the permutahedron and previous computations of Stembridge \cite{stembridgecohomology} to compute the equivariant Ehrhart theory of $\Poly_{n-1}$, and therefore shows that $\Poly_{n-1}$ satisfies \cite[Conjecture 12.1]{stapledonequivariant}. We do so directly using the fixed polytopes. 
Given $\sigma\in S_n$, we describe this action using the $(n-1)\times(n-1)$ matrix $M_\sigma$, defined by
\begin{align*}
	({M_\sigma})_{ij} = 
	\begin{cases}
		1 &\text{ if } \sigma(j) = i, \\
		-1 &\text{ if } \sigma(j) = n, \\
		0 &\text{ otherwise.}
	\end{cases}
\end{align*}
Note that for a vertex $\bu_{a_1\dots a_m}$ of $\Poly_{n-1}$,
\begin{equation*}
	M_\sigma\cdot \bu_{a_1\dots a_m} = \bu_{\sigma(a_1)\dots \sigma(a_m)},
\end{equation*}
that is, the $S_n$ action applied to a vertex can be described as $\sigma\in S_n$ shuffling the elements of the set labeling that vertex. 
For example, the transposition $(23)$ in $\Poly_{2}$ fixes the vertices $\bu_1$ and $\bu_{23}$, switches $\bu_{2}$ and $\bu_{3}$, and switches $\bu_{12}$ and $\bu_{13}$. One can verify using $M_{(23)}$ that $\Poly_{2}^{(23)}$ is a line segment, as shown in \Cref{fixed23}.

By definition, if $M_\sigma\cdot \bu_I=\bu_J$, then $I$ and $J$ have the same size.  
Moreover, given $I = \{a_1,\dots,a_m\}$ and $J = \{b_1,\dots,b_m\}$, there exists a permutation $\sigma$ that sends $\bu_I$ to $\bu_J$. Thus the orbits of this $S_n$ action are of the form $\{\bu_I \mid |I| = m\}$ for each $1\leq m<n$.

\begin{figure}
    \centering
    \includegraphics[height=1.5in]{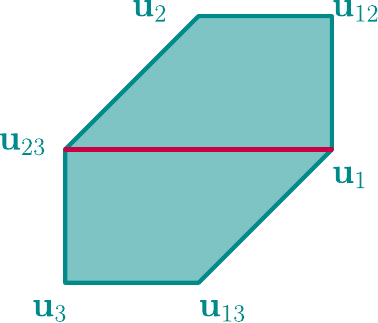}
    \hspace{1cm}
    \includegraphics[height=1.8in]{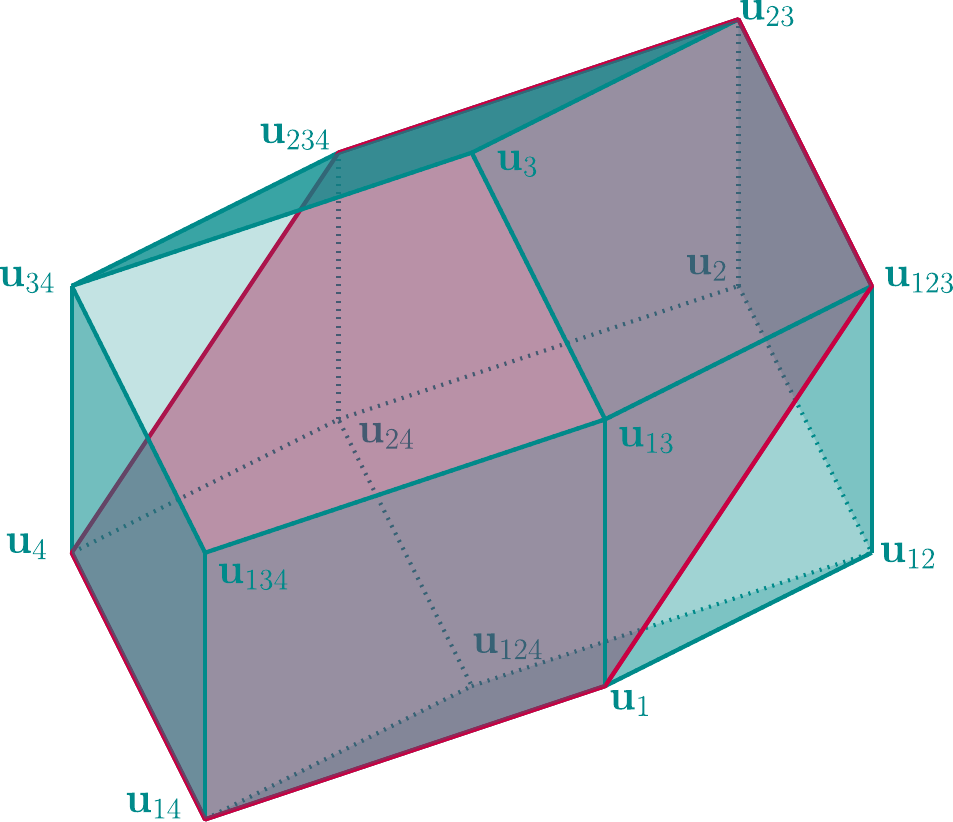}
    \caption{The fixed polytopes $\Poly_{2}^{(23)}$ and $\Poly_{3}^{(23)}$.}
    \label{fixed23}
\end{figure}

We next compute the polynomial $H^*[z](\sigma)$. Let $\sigma\in S_n$ have cycle decomposition $\sigma = \alpha_1\cdots\alpha_k$, where $\alpha_k$ contains $n$. Let $\walpha_i$ be the set of elements of $[n]$ permuted by the cycle $\alpha_i$.

\begin{remark}
    The equivariant Ehrhart theory of graphic zonotopes is computed in \cite{equivarianttechniques}. While $\Poly_{n-1}$ is a graphic zonotope, the action we describe is different than the action described in \cite{equivarianttechniques}. 
    We expand on this in \Cref{sec_2eet}.
\end{remark}

\begin{theorem}\label{thm_fixedpolytopes}
	Let $\Poly_{n-1}^\sigma$ be the fixed polytope of $\sigma\in S_n$, where $\sigma$ has cycle decomposition $\sigma = \alpha_1\cdots\alpha_k$. Then $\Poly_{n-1}^\sigma = \Zono(\textbf{u}_{\walpha_1},\dots,\textbf{u}_{\walpha_k})$. 
    As a consequence, its Ehrhart polynomial is
	\begin{align*}
    	L_{\Poly_{n-1}^\sigma}(t) = \sum_{j=0}^{k-1}\binom{k}{j} t^j = (t+1)^k-t^k
	,\end{align*}
	and its $h^*$-polynomial is the $k$th Eulerian polynomial $A_k(z)$.
\end{theorem}

\begin{proof}
	Consider vertex $\bu_{I}$ of $\Poly_{n-1}$ and $\sigma\in S_n$, where $\sigma$ has cycle decomposition $\sigma = \alpha_1\cdots\alpha_k$; assume $n$ appears in the cycle $\alpha_k$. In order to determine if $\bu_I$ is fixed by $\sigma$, we consider what each cycle does to~$\bu_I$. 
	 \begin{itemize}
	     \item 
	 If $\walpha_i\subseteq I$, then $\alpha_i(\bu_I) = \bu_{I}$ since the elements being permuted are fully contained in $I$.
		\item 
	 If $\walpha_i\cap I = \emptyset$, then $\alpha_i(\bu_I) = \bu_I$ since none of the elements being permuted are contained in $I$.
		\item 
	 If $\walpha_i\cap I \neq\emptyset$ and $\walpha_i\not\subseteq I$, then $\alpha_i(\bu_I)\neq\bu_I$ since some elements of $I$ will be sent to elements not in $I$ and some elements not in $I$ will be sent to elements in $I$.
	 \end{itemize}
	Putting these together, we can see that $\sigma(\bu_I) = \bu_I$ exactly when for each $i$, either $\walpha_i\subseteq I$ or $\walpha_i\cap I = \emptyset$. Since $\bu_I = \sum_{j\in I}\bu_j$, this also tells us that $\Zono(\bu_{\walpha_1},\dots,\bu_{\walpha_k})\subseteq\Poly_{n-1}^\sigma$.
	
	For the opposite containment, observe that $\Poly_{n-1}^\sigma$ is contained in the hyperplanes defined, for $i<k$, by $x_a = x_b$ if $a,b\in\walpha_i$, and $x_c = 0$ for $c\in\walpha_k\setminus\{n\}$. To see why $x_c=0$, we consider two cases.
	First, assume $|\walpha_k|=2$ and write $\alpha_k = (j n)$. 
	The $j$th row of the corresponding matrix $M_\sigma$ will be $-\be_j$, i.e. all zeroes except a $-1$ at the $j$-th entry.
	Applying $M_\sigma$ to a point $p = (p_1, \dots, p_{n-1})\in\RR^{n-1}$, then, will map $p_j$ to $-p_j$. In order for $p$ to be fixed by $M_\sigma$, it must be that $p_j = 0$. 
    
    Suppose, then, that $|\walpha_k| = m>2$; say $\alpha_k = (a_1 a_2 \dots a_{m-1} n)$. The $a_i$th column of the corresponding matrix $M_\sigma$ is $\be_{a_{i+1}}$ and the $a_{m-1}$st column is $-\mathbb{1}$. Applying $M_\sigma$ to $p\in\RR^{n-1}$ yields
    \begin{align*}
    	p_{a_1} &\mapsto -p_{a_{m-1}} \\
    	p_{a_i} &\mapsto p_{a_{i-1}} - p_{a_{m-1}} \text{ for } i\in\{2,\dots,m-1\}.
    \end{align*}
    Solving the system of equations obtained from $p$ being fixed yields $p_{a_i} = 0$ for all $i\in[m-1]$.
    We conclude that for $p = (p_1,\dots,p_{n-1})\in\Poly_{n-1}^\sigma$, $p_a = p_b$ whenever $a,b\in\walpha_i$ for some $i$. Denote by $p_{\walpha_i}$ the value of $p_a$ for $a\in\walpha_i$.

	Now, we need to show that $p = \sum_{i=1}^{k}\lambda_i \bu_{\walpha_i}$ for $0\leq\lambda_i\leq 1$. Since $p\in\Poly_{n-1}$, we know that $0\leq |p_j|\leq 1$ for each $j$. Next, we know that the sum of the $\bu_i$s is 0; observe, then, since the $\walpha_i$s form a partition of $[n]$, that $\bu_{\walpha_k} = -\sum_{i=1}^{k-1}\bu_{\walpha_i}$. It is clear that $p = \sum_{i=1}^{k-1} p_{\walpha_i}\bu_{\walpha_i}$. If each $p_{\walpha_i}\geq 0$, we are done; suppose, then, that this is not the case. Let $\lambda_k = \max_{i\leq k-1} \{-p_{\walpha_i}\}$ and for $i<k$, let $\lambda_i = p_{\walpha_i} + \lambda_k$. Since $-\lambda_k\leq p_{\walpha_i}$ for each $i$, $\lambda_i\geq 0$. Moreover, points in $\Poly_{n-1}$ satisfy $x_j-x_i\leq 1$ and $x_j-x_i\geq -1$, for $i<j$, so $\lambda_i\leq 1$. Then	
	\begin{align*}
		p = \sum_{i=1}^{k} \lambda_i \bu_{\walpha_i}.
	\end{align*}
	Hence $\Poly_{n-1}^\sigma = \Zono(\bu_{\walpha_1},\dots,\bu_{\walpha_k})$. Moreover, observe that $\bu_{\walpha_1},\dots,\bu_{\walpha_{k-1}}$ are linearly independent. Since each $\bu_{\walpha_i}$ is a primitive lattice vector and $\bu_{\walpha_k}=-\sum_{i=1}^{k-1}\bu_{\walpha_i}$, $\Poly_{n-1}^\sigma$ is unimodularly equivalent to $\Poly_{k-1}$. 

    The second claim follows from Proposition \ref{ehrdA}.
\end{proof}

We now prove directly a result that can be recovered from Stapledon \cite[Prop 8.1]{stapledonequivariant} and Stembridge \cite[Cor 6.1]{stembridgecohomology}.

\begin{corollary}\label{itsapolynomial}
    	Let $\sigma\in S_n$ have cycle decomposition $\sigma = \alpha_1\cdots\alpha_k$. The equivariant $H^*$-series of $\Poly_{n-1}$ is given by
	\begin{align*}
    	H^*[z](\sigma) = A_k(z)\prod_{j=1}^{k} (1+z+\dots+z^{|\walpha_j|-1}).
	\end{align*}
\end{corollary}
\begin{proof}
	Let $\sigma\in S_n$. 
	Then
	\begin{align*}
		\det(I - M_\sigma\cdot z) &= \frac{\prod_{i=1}^{k}\left(1-z^{|\walpha_i|}\right)}{1-z} = \frac{\prod_{i=1}^{k}(1-z)\left(1+z+\dots +z^{|\walpha_i|-1}\right) }{1-z}\\
		&= (1-z)^{k-1}\prod_{i=1}^{k}\left(1+z+\dots +z^{|\walpha_i|-1}\right).
	\end{align*}
	Now, we know that the Ehrhart series of $\Poly_{n-1}^\sigma$ is
	\begin{align*}
		\Ehr_{\Poly_{n-1}^\sigma}(z) = \frac{A_k(z)}{(1-z)^{k}}.
	\end{align*}
	Since the denominator of the equivariant Ehrhart series of $\sigma$ has an extra factor of $1-z$, the only difference is the product of sums of powers of $z$. In particular,
	\begin{align*}
		H^*[z](\sigma) &= A_k(z)\prod_{i=1}^{k}\left(1+z+\dots +z^{|\walpha_i|-1}\right). \qedhere
	\end{align*}
\end{proof}

This equivariant $H^*$-series shows up elsewhere, for instance in \cite[Proposition 6.6]{shareshianwachseulqu}, as the Frobenius characteristic of the permutation representation of the toric variety associated to the permutahedron.
	\begin{table}
	\begin{center}
	\begin{tabular}{|c|c|c|c|}
		\hline
		Cycle type of $\sigma\in S_4$ & $L_{\Poly_{3}^\sigma}(t)$ & $\displaystyle\sum_{t\geq 0} \chi_{t\Poly_3}(g) z^t$ & $H^*[z](\sigma)$ \\
		\hline & & & \\[-6pt]
		$(1,1,1,1)$ & $4t^3+6t^2+4t+1$ & $\frac{(1+11z+11z^2+z^3)(1)}{(1-z)^4}$ & $1+11z+11z^2+z^3$ \\[6pt]
		\hline & & & \\[-6pt]
		$(2,1,1)$ & $3t^2+3t+1$ & $\frac{(1+4z+z^2)(1+z)}{(1-z)(1-z)(1-z^2)}$ & $1+5z+5z^2+z^3$ \\[6pt]
		\hline & & & \\[-6pt]
		$(2,2)$ & $2t+1$ & $\frac{(1+z)(1+z)(1+z)}{(1-z^2)(1-z^2)}$ & $1+3z+3z^2+z^3$ \\[6pt]
		\hline & & & \\[-6pt]
		$(3,1)$ & $2t+1$ & $\frac{(1+z)(1+z+z^2)}{(1-z)(1-z^3)}$ & $1+2z+2z^2+z^3$ \\[6pt]
		\hline & & & \\[-6pt]
		$(4)$ & 1 & $\frac{(1)(1+z+z^2+z^3)}{1-z^4}$ & $1+z+z^2+z^3$ \\[6pt]
		\hline
	\end{tabular}
	\end{center}
	\caption{The equivariant $H^*$-series of $\Poly_{3}$.}
	\end{table}


\subsection{Two equivariant Ehrhart theories} \label{sec_2eet}

In \cite[Section 3.1]{equivarianttechniques} the equivariant Ehrhart theory of $\Zono_{C_n}$ is studied with respect to a particular group action. The automorphism group of the graph, $\Aut(C_n)$, acts on $[n]$ to shuffle the nodes of the graph while leaving the edge set invariant, which induces a linear action on $\RR^{n}$ under which $\Zono_{C_n}$ is invariant. Observe that $\Aut(C_n)\neq S_n$: consider the transposition $\sigma = (23)$ acting on $C_4$. Permuting nodes $2$ and $3$ changes the edge set by replacing edges $12$ and $34$ with edges $13$ and $24$, as seen in \Cref{cyclegraph23}.
\begin{figure}[h!]
	\centering
	\includegraphics[width=0.26\textwidth]{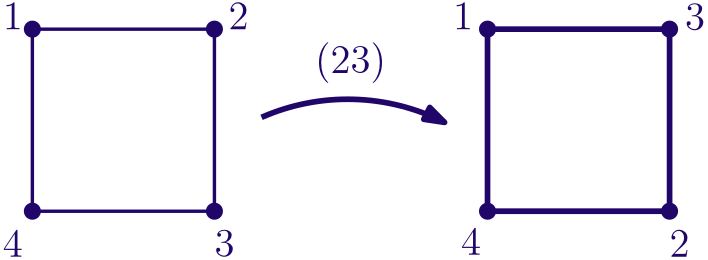}
	\caption{$C_4$ with $(23)$ applied.}
	\label{cyclegraph23}
\end{figure}

Now, by \Cref{prop: zono is graph zono}, this action by $\Aut(C_n)$ also gives an equivariant Ehrhart theory for $\Poly_{n-1}$. The $\Aut(C_n)$ action is very different from the $S_n$ action we have already studied. As an example, consider $n=4$ and $\sigma = (24)$. By \Cref{thm_fixedpolytopes}, we know that $\Poly_{4-1}^{\sigma}$ is a hexagon, since $\sigma = (1)(3)(24)$. On the other hand, \cite[Theorem 3.11]{equivarianttechniques} tells us that the fixed polytope under the $\Aut(C_n)$ action has as many vertices as there are acyclic orientations of the $\sigma$-connectivity graph $C_{C_4}(\sigma)$, shown in \Cref{24acyclic}.
\begin{figure}[h!]
	\centering
	\includegraphics[width=0.17\textwidth]{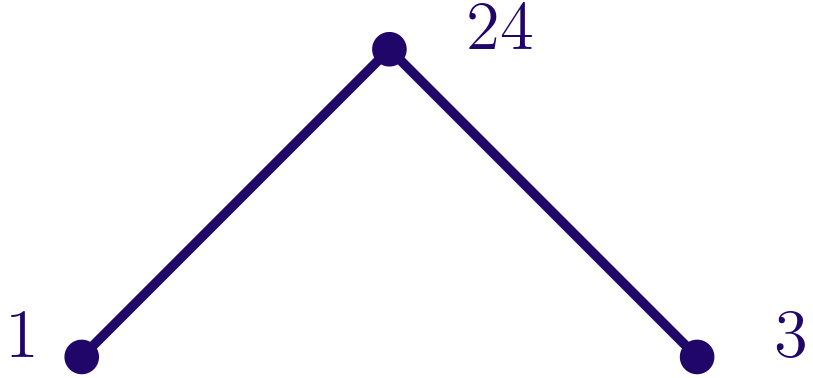}
	\caption{The $\sigma$-connectivity graph $C_{C_4}(\sigma)$.}
	\label{24acyclic}
\end{figure}
As there are 4 acyclic orientations of $C_{C_4}(\sigma)$, the corresponding fixed polytope has 4 vertices, not 6. The two different fixed polytopes are shown in \Cref{24fixed}.
\begin{figure}[h!]
	\centering
	\includegraphics[width=0.3\textwidth]{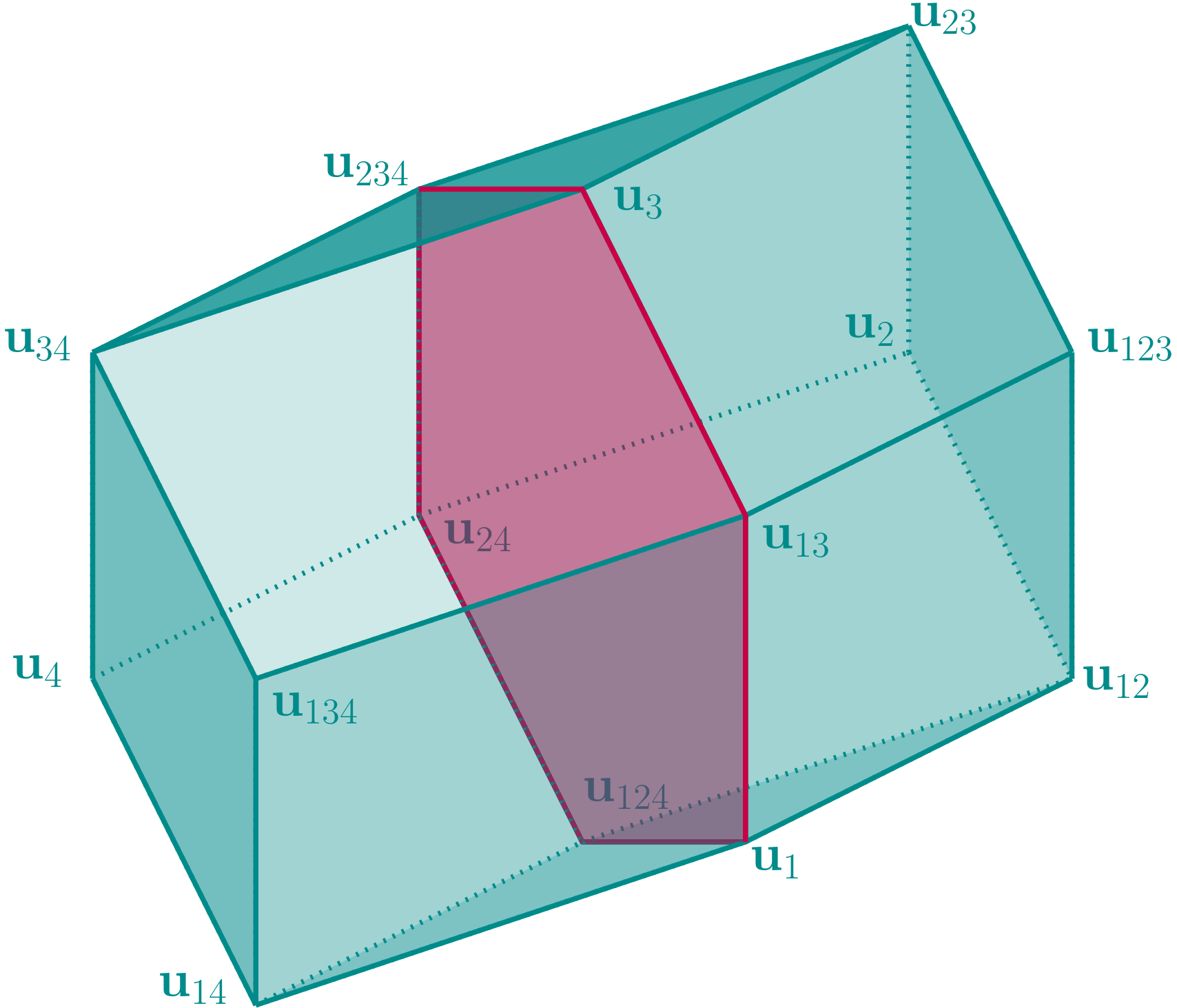}
	\hspace{24pt}
	\includegraphics[width=0.3\textwidth]{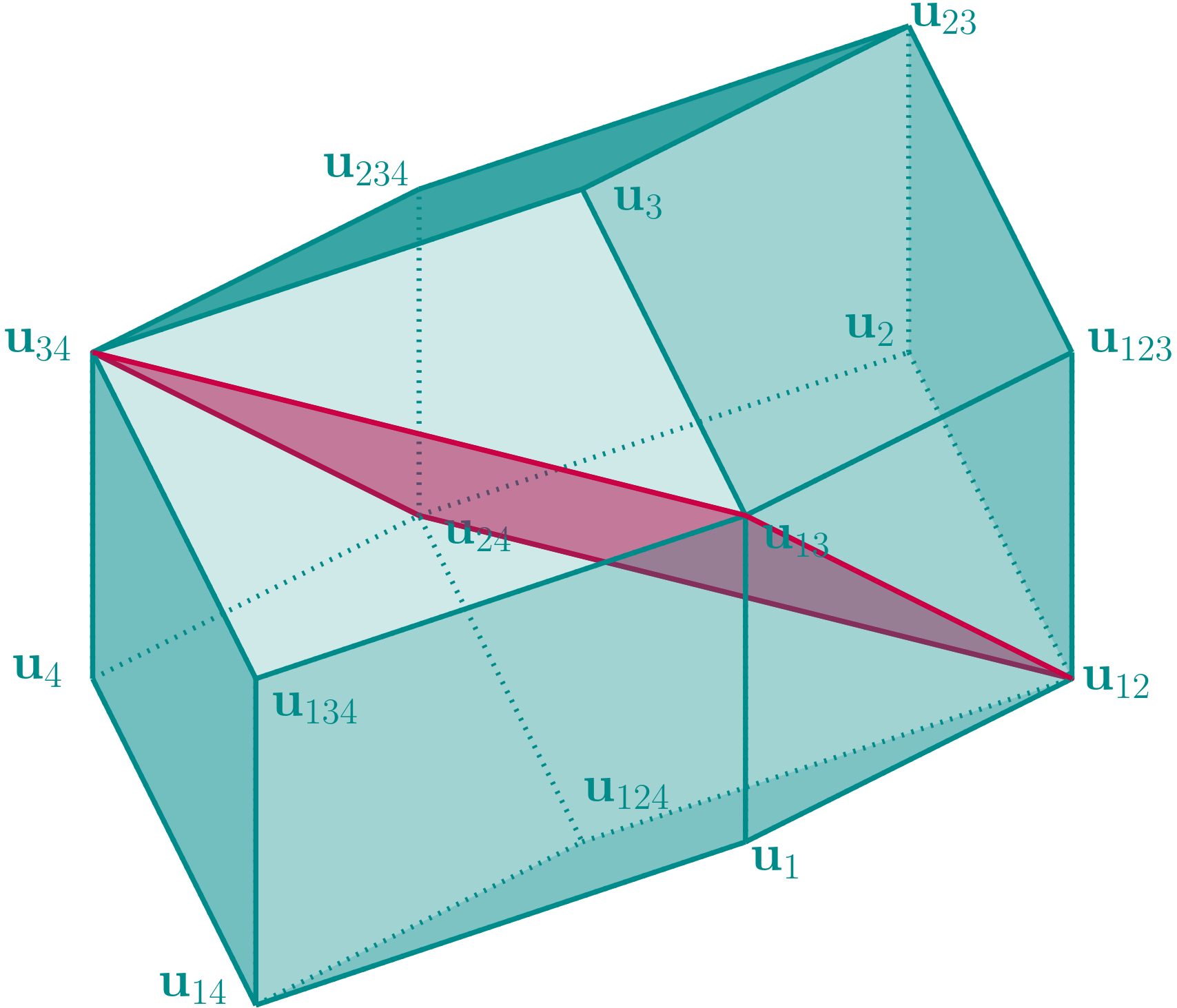}
	\caption{$\Poly_{3}^{(24)}$ versus $\Zono_{C_4}^{(24)}$.}
	\label{24fixed}
\end{figure}


\section*{Acknowledgments}

We would like to thank the organizers of the REACT 2021 workshop which led us to these polytopes. We are also grateful to Matthias Beck for posing this problem and to Matthias Beck and John Shareshian for many fruitful conversations.
Laura Escobar was partially supported by NSF Grant DMS-1855598 and NSF CAREER Grant DMS-2142656, DMS-2521270. Jodi McWhirter was partially supported by NSF Grant DMS-1855598.


\small

\bibliographystyle{amsplain}
\bibliography{bib}

\providecommand{\bysame}{\leavevmode\hbox to3em{\hrulefill}\thinspace}
\providecommand{\MR}{\relax\ifhmode\unskip\space\fi MR }
\providecommand{\MRhref}[2]{%
  \href{http://www.ams.org/mathscinet-getitem?mr=#1}{#2}
}
\providecommand{\href}[2]{#2}
\begin{thebibliography}{10}

\bibitem{AguiarArdila}
Marcelo Aguiar and Federico Ardila, \emph{Hopf monoids and generalized
  permutahedra}, Mem. Amer. Math. Soc. \textbf{289} (2023), no.~1437, vi+119.

\bibitem{ardilaschindlervindas}
Federico Ardila, Anna Schindler, and Andr\'es~R. Vindas-Mel\'endez, \emph{The
  equivariant volumes of the permutahedron}, Discrete Comput. Geom. \textbf{65}
  (2021), no.~3, 618--635. \MR{4226485}

\bibitem{ardilasupinavindas}
Federico Ardila, Mariel Supina, and Andr\'{e}s~R. Vindas-Mel\'{e}ndez,
  \emph{The equivariant {E}hrhart theory of the permutahedron}, Proc. Amer.
  Math. Soc. \textbf{148} (2020), no.~12, 5091--5107.

\bibitem{BG}
F.~Barahona and M.~Gr\"otschel, \emph{On the cycle polytope of a binary
  matroid}, J. Combin. Theory Ser. B \textbf{40} (1986), no.~1, 40--62.
  \MR{830592}

\bibitem{orientedmatroids}
Anders Bj{\"o}rner, Michel Las~Vergnas, Bernd Sturmfels, Neil White, and
  G{\"u}nter~M. Ziegler, \emph{Oriented matroids}, second ed., Encyclopedia of
  Mathematics and its Applications, vol.~46, Cambridge University Press,
  Cambridge, 1999.

\bibitem{ClaHigKol}
Oliver Clarke, Akihiro Higashitani, and Max K\"olbl, \emph{The equivariant
  {E}hrhart theory of polytopes with order-two symmetries}, Proc. Amer. Math.
  Soc. \textbf{151} (2023), no.~9, 4027--4041.

\bibitem{criadotropicalbisectors}
Francisco Criado, Michael Joswig, and Francisco Santos, \emph{Tropical
  bisectors and {V}oronoi diagrams}, Found. Comput. Math. \textbf{22} (2022),
  no.~6, 1923--1960.

\bibitem{manyfacessep}
Alessio D'Al{\`i}, Emanuele Delucchi, and Mateusz Micha\l{}ek, \emph{Many faces
  of symmetric edge polytopes}, Electron. J. Combin. \textbf{29} (2022), no.~3,
  Paper No. 3.24, 42.

\bibitem{sepgentoregmats}
Alessio D'Al{\`i}, Martina Juhnke-Kubitzke, and Melissa Koch, \emph{On a
  generalization of symmetric edge polytopes to regular matroids},
  International Mathematics Research Notices \textbf{2024} (2024), no.~14,
  10844--10864.

\bibitem{deloerahawskoeppe}
Jes\'{u}s~A. De~Loera, David~C. Haws, and Matthias K\"{o}ppe, \emph{Ehrhart
  polynomials of matroid polytopes and polymatroids}, Discrete Comput. Geom.
  \textbf{42} (2009), no.~4, 670–702.

\bibitem{deloerahemmeckekoeppe}
Jes{\'u}s~A. De~Loera, Raymond Hemmecke, and Matthias K{\"o}ppe,
  \emph{Algebraic and {G}eometric {I}deas in the {T}heory of {D}iscrete
  {O}ptimization}, MOS-SIAM Series on Optimization, vol.~14, Society for
  Industrial and Applied Mathematics (SIAM), Philadelphia, PA; Mathematical
  Optimization Society, Philadelphia, PA, 2013.

\bibitem{ehrhartpolynomial}
Eug{\`e}ne Ehrhart, \emph{Sur les poly\`edres rationnels homoth\'etiques \`a
  {$n$}\ dimensions}, C. R. Acad. Sci. Paris \textbf{254} (1962), 616--618.

\bibitem{equivarianttechniques}
Sophia Elia, Donghyun Kim, and Mariel Supina, \emph{Techniques in equivariant
  {E}hrhart theory}, Ann. Comb. \textbf{28} (2024), no.~3, 819--870.

\bibitem{Ferroni2024}
Luis Ferroni and Akihiro Higashitani, \emph{Examples and counterexamples in
  {E}hrhart theory}, EMS Surveys in Mathematical Sciences (2024).

\bibitem{FinschiCatalogOM}
Lukas Finschi, \emph{Catalog of isomorphism classes of oriented matroids},
  \url{https://finschi.com/math/om/?p=catom}, Accessed: 2026-07-07.

\bibitem{FinschiFukuda2002}
Lukas Finschi and Komei Fukuda, \emph{Generation of oriented matroids---a graph
  theoretical approach}, Discrete \& Computational Geometry \textbf{27} (2002),
  117--136.

\bibitem{Gru}
Vladimir Gruji\'{c}, \emph{Counting faces of graphical zonotopes}, Ars Math.
  Contemp. \textbf{13} (2017), no.~1, 227--234. \MR{3669576}

\bibitem{higashitanikummermichalek}
Akihiro Higashitani, Mario Kummer, and Mateusz Micha{\l}ek, \emph{Interlacing
  {E}hrhart polynomials of reflexive polytopes}, Selecta Math. (N.S.)
  \textbf{23} (2017), no.~4, 2977--2998. \MR{3703472}

\bibitem{antisymmetricflows}
Winfried Hochst\"attler and Jaroslav Ne\u~set\u ril, \emph{Antisymmetric flows
  in matroids}, European J. Combin. \textbf{27} (2006), no.~7, 1129--1134.

\bibitem{shareshianwachseulqu}
John Shareshian and Michelle~L. Wachs, \emph{Eulerian quasisymmetric
  functions}, Adv. Math. \textbf{225} (2010), no.~6, 2921--2966.

\bibitem{stanleyec1}
Richard~P. Stanley, \emph{Enumerative {C}ombinatorics. {V}olume 1}, second ed.,
  Cambridge Studies in Advanced Mathematics, vol.~49, Cambridge University
  Press, Cambridge, 2012.

\bibitem{stapledonequivariant}
Alan Stapledon, \emph{Equivariant {E}hrhart theory}, Advances in Mathematics
  \textbf{226} (2011), no.~4, 3622--3654.

\bibitem{stembridgecohomology}
John~R. Stembridge, \emph{Some permutation representations of {W}eyl groups
  associated with the cohomology of toric varieties}, Adv. Math. \textbf{106}
  (1994), no.~2, 244--301.

\end{thebibliography}

\setlength{\parskip}{0cm} 

\end{document}